\documentclass[nosumlimits,twoside]{amsart}
\usepackage{amsfonts, amsmath, amssymb}
\usepackage[bookmarksnumbered,plainpages,driverfallback=dvipdfm,backref=page]{hyperref}
\usepackage{graphicx}
\usepackage{float}
\usepackage{srcltx}
\usepackage[all]{xy}
\usepackage{version}
\usepackage[final]{showlabels} 
\usepackage[T1]{fontenc}
\usepackage{enumerate}
\usepackage[normalem]{ulem}
\usepackage{bm}
\usepackage{latexsym}
\usepackage[2emode]{psfrag}
\usepackage{yhmath}
\usepackage{array}
\usepackage{dsfont}
\usepackage{mathrsfs}
\usepackage{tikz-cd}
\usetikzlibrary{trees, arrows, arrows.meta, shapes.geometric, automata}
\usepackage{comment}
\usepackage{xcolor}
\usepackage{aligned-overset}

\newtheorem{theorem}{\sc Theorem}[section]
\newtheorem{proposition}[theorem]{\sc Proposition}
\newtheorem{notation}[theorem]{\sc Notation}
\newtheorem{lemma}[theorem]{\sc Lemma}
\newtheorem{corollary}[theorem]{\sc Corollary}
\theoremstyle{definition}
\newtheorem{definition}[theorem]{\sc Definition}
\newtheorem{remark}[theorem]{\sc Remark}

\allowdisplaybreaks
\excludeversion{invisible}

\newcommand{\chara}{\textup{char }}

\newcommand{\sgn}{\textup{sgn}}

\begin{document}


\title[Separable cowreaths in higher dimension]{Separable cowreaths in higher dimension}

\author{Fabio Renda}
\date{\today}

\address{\parbox[b]{0.97\linewidth}{University of Ferrara, Department of Mathematics and Computer Science,\\ Via Machiavelli 30, 44121 Ferrara, Italy}}
\email{fabio.renda@unife.it}



\begin{abstract}
In this paper we present an infinite family of (h-)separable cowreaths with increasing dimension. Menini and Torrecillas proved in \cite{mt2} that for $A=Cl(\alpha,\beta, \gamma)$, a four-dimensional Clifford algebra, and $H=H_4$, Sweedler's Hopf algebra, the cowreath $(A \otimes H^{op},H, \psi)$ is always (h-)separable. We show how to produce similar examples in higher dimension by considering a $2^{n+1}$-dimensional Clifford algebra $A=Cl(\alpha,\beta_i,\gamma_i,\lambda_{ij})$ and $H=E(n)$, a suitable pointed Hopf algebra that generalizes $H_4$. We adopt the approach pursued in \cite{mt}, requiring that the separability morphism be of a simplified form, which in turn forces the defining scalars $\alpha,\beta_i,\gamma_i,\lambda_{ij}$ to satisfy further conditions. 
\end{abstract}

\maketitle

\tableofcontents

\section{Introduction}

This paper is closely related to \cite{mt,mt2} in that it provides further results on the separability of cowreaths constructed on two-sided Hopf modules (see also \cite[Sec. 3]{bct2}). The concept of separable algebra is a generalization of that of separable field extension and as such was first defined in the context of vector spaces (cf. \cite{AG}). It is usually said that a $\Bbbk$-algebra $A$ is separable if the multiplication map $m_A: A \otimes A \rightarrow A$ admits a section $\gamma: A \rightarrow A \otimes A$ which is also a morphism of $A$-$A$-bimodules. This is easily generalized to the case of algebras in a monoidal category. Moreover, when adopting a categorical approach, separability of an object can be interpreted also as a property of a certain forgetful functor. In order to do this, one needs the notion of separable functor, introduced in \cite{nbo}. The importance of the concept of separability and its widespread applications are evidenced in \cite{Fo}. The dual notion of coseparable coalgebra was first presented by Larson in \cite{Lar} in terms of module injectivity, and was later shown to represent a suitable dualization of all the equivalent characterizations of separability (see \cite[Prop.~7.3]{bct}, \cite[Thm.~2.5]{mt}). 

Ardizzoni and Menini introduced in \cite{am} a stronger version of the notion of separability for functors, called heavy separability (or, for short, h-separability). This led to the discovery of a Rafael-type theorem for h-separable functors \cite[Theorem 2.1]{am} and to the definition of h-coseparable coalgebra \cite[Def. 2.2]{mt}. Unfortunately, it was proved that the only h-coseparable coalgebras $C$ over a field $\Bbbk$ are $C=\Bbbk$ and $C=\lbrace 0 \rbrace$. Non-trivial examples of h-coseparable coalgebras were later exhibited in \cite{mt} as objects of the category $\mathcal{T}_{A \otimes H^{op}}^{\#}$ of right transfer morphism through $A \otimes H^{op}$, where $A=Cl(\alpha,\beta, \gamma)$ is a four-dimensional Clifford algebra and $H=H_4$ is Sweedler's Hopf algebra. These objects are realized as instances of particular entwined structures called cowreaths, which were introduced by Brzezinski in \cite{br} in order to develop Galois theory for coalgebras. In detail, given an algebra $A$ in a monoidal category $\mathcal{M}$, one can construct a new category
$\mathcal{T}_A^{\#}$, where objects are pairs $(X, \psi)$, where $X$ is an object in $\mathcal{M}$ and $\psi: X \otimes A \rightarrow A \otimes X$ is a morphism in $\mathcal{M}$ compatible with the algebra structure of $A$, and a triple $(A, X, \psi)$ is called a cowreath in $\mathcal{M}$ if $(X, \psi)$ is a coalgebra in $\mathcal{T}_A^{\#}$. The theory of separable cowreaths was developed in \cite{mt}.

When $H$ is a Hopf algebra over a field $\Bbbk$ and $A$ is a right $H$-comodule algebra, we find that $A \otimes H^{op}$ is naturally a right $H \otimes H^{op}$-comodule algebra and $H$ is a $H \otimes H^{op}$-module coalgebra and we can construct a cowreath $(A \otimes H^{op}, H, \psi)$ in the category of $\Bbbk$-vector spaces. The examples of h-separable cowreaths presented in \cite{mt} are built in view of the fact that a four-dimensional Clifford algebra $Cl(\alpha,\beta, \gamma)$ admits a canonical comodule-algebra structure over Sweedler's Hopf algebra $H_4$. Indeed, it was shown in \cite{PVO2} that, more in general, a $2^{n+1}$-dimensional Clifford algebra $A=Cl(\alpha, \beta_i, \gamma_i, \lambda_{ij})$ admits a canonical $E(n)$-comodule algebra structure $\rho: A \rightarrow A \otimes E(n)$ that makes it a cleft extension of the groud field $\Bbbk$. The name $E(n)$ is used to denote a family of pointed Hopf algebras that generalizes Sweedler's Hopf algebra ($H_4=E(1)$) and that were introduced in \cite{BDG} and studied in \cite{CC, CD, PVO, PVO2}. Every $E(n)$-comodule algebra structure $\rho:A \rightarrow A \otimes E(n)$ can be described in terms of a tuple of maps $(\varphi, d_1, \ldots, d_n)$ satisfying suitable properties, as it was shown in \cite{FR2}. Then the separability of a cowreath $(A \otimes E(n)^{op}, E(n), \psi)$ can also be characterized by a list of equalities involving these maps, if we further assume the Casimir element realizing separability to be of a simplified form (see Theorem~\ref{theo1}). This allows us to exhibit a whole family of (h-)separable cowreaths in higher dimension that generalizes the example described in \cite{mt}.

\medskip

The article is organized as follows. 
In Section \ref{preliminaries} we recall some basics on cowreaths in monoidal categories and their (h-)separability, we introduce the definition of the Hopf algebras $E(n)$ and we describe its comodule-algebras, among which Clifford algebras $Cl(\alpha, \beta_i, \gamma_i, \lambda_{ij})$ play a special role. Section~\ref{technicalresults} contains a series of technical lemmas that will help us with calculations in the sequel. Some of these lemmas deal with partitions of a set, some with a function controlling the sign of a product in Clifford algebras and some other illustrate useful properties of a family of distinguished skew-derivations. In Section~\ref{rtseparability} we show that (h-)separability of the cowreath $(A \otimes E(n)^{op}, E(n), \psi)$, with $A$ a finite-dimensional $E(n)$-comodule algebra, can be characterized by a family of equalities involving a tuple of maps related to the comodule algebra structure of $A$ \--- provided that the Casimir element $B$ that realizes separability is of a simplified form. Finally, in Section~\ref{Cliffordcase} we prove that when  $A=Cl(\alpha,\beta_i,\gamma_i, \lambda_{ij})$ is a Clifford algebra endowed with its canonical $E(n)$-coaction, the aforementioned list of equalities admits solutions for every $n$.

\medskip

\noindent\textit{Notations and conventions}. All vector spaces are understood to be over a fixed field $\Bbbk$ and by linear maps we mean $\Bbbk$-linear maps. $\Bbbk^{\times}$ is used to indicate the multiplicative group of the field $\Bbbk$. The unadorned tensor product $\otimes$ stands for $\otimes_{\Bbbk}$. We will write a tensor symbol before the exponent when denoting iterated tensor products: $X^{\otimes 2}=X \otimes X$, $X^{\otimes 3}=X \otimes X \otimes X$ and so on. All linear maps whose domain is a tensor product will usually be defined on generators and understood to be extended by linearity. 
Algebras over $\Bbbk$ will be associative and unital and coalgebras over $\Bbbk$ will be coassociative and counital. For an algebra, the multiplication and the unit are denoted by $m$ and $u$, respectively, while for a coalgebra, the comultiplication and the counit are denoted by $\Delta$ and $\varepsilon$, respectively. $S:H \rightarrow H$ denotes the antipode map of $H$. We use the classical Sweedler’s notation 
and we shall write $\Delta(h)= h_{1}\otimes h_{2}$ for any $h\in H$, where we omit the summation symbol. $H^{op}$ denotes the opposite Hopf algebra, i.e. the Hopf algebra $H$ where the multiplication is replaced by $m^{op}(h \otimes h'):=h'h$ for every $h,h' \in H$. For a (right) $H$-comodule $M$ we write its structure map as $\rho: M\to M\otimes H$, $\rho(m)=m_0\otimes m_1$ for every $m \in M$. We use $\sqcup$ to denote the disjoint union of two sets and we write $Q^c$ to indicate the complement of the set $Q$ in $\lbrace 1, \ldots, n \rbrace$. We denote by $\delta_{ij}$ the usual Kronecker symbol. We sometimes use it in a broader sense, writing $\delta_{p}$ \--- where $p$ is any proposition depending on a set of variables \--- with the convention that $\delta_{p}$ represent $1$ if $p$ is true and $0$ otherwise (e.g. $\delta_{i \in Q}$ is $1$ if $i \in Q$ and $0$ otherwise). $\lfloor m \rfloor$ denotes the biggest integer less than or equal to $m$. 

\section{Preliminaries}\label{preliminaries}
Throughout the article we will mostly work in the category of vector spaces over a fixed field $\Bbbk$, but since the natural setting for cowreaths is that of monoidal categories, we will adopt a more general approach to present the first definitions. If needed, a solid reference for basics of category theory is \cite{bo1, bo2}. After a second part where we introduce the notion of (h-)separable cowreaths, we restrict to the category $\mathrm{Vec}_{\Bbbk}$ of vector spaces over $\Bbbk$, presenting some results on separability in this setting. Finally we recall the definitions of our two main ingredients: the Hopf algebras $E(n)$ and the Clifford algebras $Cl(\alpha,\beta_i, \gamma_i, \lambda_{ij})$. Most of the contents of this introductory section are taken from \cite{bct, mt,FR2}. Proofs of results are omitted, the reader interested in details is referred to these papers.

\subsection{Cowreaths}\label{cowreathsec}
Wreaths and cowreaths are generalized entwined structures that can be introduced in the context of $2$-categories. The definition of wreath was given in \cite[p.256]{ls}, while the dual notion of cowreath in \cite[p. 1056]{bc}. 
\begin{definition}{\cite{ls,bc}}
Let $\mathcal{M}$ be a $2$-category. Then a (co)wreath in $\mathcal{M}$ is a (co)monad in the Eilenberg-Moore category $\textrm{EM}(\mathcal{M})$ of $\mathcal{M}$. 
\end{definition}
In \cite{bc} the authors show that when $\mathcal{M}$ is a monoidal category, a cowreath in $\mathcal{M}$ is equivalently defined by the choice of an algebra $A$ in $\mathcal{M}$, an object $X$ in $\mathcal{M}$, and an entwining map $\psi: X \otimes A \rightarrow A \otimes X$ satisfying appropriate relations. Among these relations is the fact that $(X, \psi)$ is a coalgebra in a suitable monoidal category associated to the algebra $A$. This ultimately leads to the notion of corepresentation of a cowreath, i.e. of entwined module over a cowreath \cite[Def.~3.2]{bct}.
Further details about (co)wreaths in a $2$-category, their (co)representations and motivations for their investigation can be found in \cite{bc, bct, bct2}.
Given that cowreaths can be regarded as coalgebras in a suitable monoidal category, the notions of separable and Frobenius cowreath can also be introduced and examples of such structures can be presented  (see \cite{bct2,mt}).

\begin{definition}\cite[Subsec.~3.1]{bct}\label{transf}
Let $\mathcal{M}$ be a (strict) monoidal category and let $(A, m, u)$ be an algebra in $\mathcal{M}$. A (right) transfer morphism through $A$ is a pair $(X, \psi)$ with $X \in \mathcal{M}$ and $\psi : X \otimes A \rightarrow A \otimes X$ in $\mathcal{M}$
such that
\begin{align*}
\psi (\textrm{Id}_X \otimes m) &= (m \otimes \textrm{Id}_X) (\textrm{Id}_A \otimes \psi) (\psi \otimes \textrm{Id}_A),\\
\psi (\textrm{Id}_X \otimes u) &= u \otimes \textrm{Id}_X.
\end{align*}
The category of all right transfer morphism through $A$ will be denoted by $\mathcal{T}_A^{\#}$ and a morphism $f:X \rightarrow Y$ for this category is a morphism $f: X \rightarrow A \otimes Y$ in $\mathcal{M}$ such that
\begin{equation*}\label{morphT}
(m \otimes \textrm{Id}_Y) (\textrm{Id}_A \otimes f) \psi_X =(m \otimes \textrm{Id}_Y) (\textrm{Id}_A \otimes \psi_Y) (f \otimes \textrm{Id}_A).
\end{equation*}
The composition of two morphisms $f : X \rightarrow Y$ and $g : Y \rightarrow Z$ in $\mathcal{T}_A^{\#}$ is
\begin{equation*}\label{compT}
g \odot f = (m \otimes \textrm{Id}_Z) (\textrm{Id}_A \otimes g) f
\end{equation*}
and we have
\[\mathrm{Id}_{(X,\psi)} = u \otimes \textrm{Id}_X.\]
The tensor product of $(X, \psi_X)$ and $(Y, \psi_Y )$ is
\begin{equation*}\label{tensorpsi}
X \circledast Y = (X \otimes Y, \ \psi_X \odot \psi_Y = (\psi_X \otimes \textrm{Id}_Y )(\textrm{Id}_X \otimes \psi_Y )).
\end{equation*} 
The tensor product of morphisms $f : X \rightarrow X'$ and $g : Y \rightarrow Y'$ is
\begin{equation*}\label{tensorfg}
f \circledast g = (m \otimes \textrm{Id}_{X'} \otimes \textrm{Id}_{Y'}) (\textrm{Id}_A \otimes \psi_X \otimes Y') (f \otimes g) .
\end{equation*}
The unit object of $\mathcal{T}_A^{\#}$ is
\[(\underline{1}, r_A^{-1} \circ l_A : \underline{1} \otimes A \rightarrow A \otimes \underline{1}).\]
\end{definition}

\begin{definition}\cite[Subsec.~3.2]{bct}\label{defcowreath}
A cowreath in $\mathcal{M}$ is a triple $(A, X, \psi)$ where $A$ is an algebra in $\mathcal{M}$ and $(X, \psi)$ is a coalgebra in $\mathcal{T}_A^{\#}$. 
This means that $(X, \psi) \in \mathcal{T}_A^{\#}$ and there are morphisms
$\delta : X \rightarrow A \otimes X \otimes X$ and $\varepsilon : X \rightarrow A$
such that 
\begin{equation*}\label{deltamor}
(m \otimes \textrm{Id}_{X^{\otimes 2}})(\textrm{Id}_A \otimes \psi \otimes \textrm{Id}_X)(\textrm{Id}_A \otimes \textrm{Id}_X \otimes \psi)(\delta \otimes \textrm{Id}_A)= (m \otimes \textrm{Id}_{X^{\otimes 2}})(\textrm{Id}_A \otimes \delta)\psi \\
\end{equation*}
($\delta$ is a morphism in $\mathcal{T}_A^{\#}$),
\begin{equation*}\label{cowcoass}
(m \otimes \textrm{Id}_{X^{\otimes 3}})(\textrm{Id}_A \otimes \delta \otimes \textrm{Id}_X)\delta = (m \otimes \textrm{Id}_{X^{\otimes 3}})(\textrm{Id}_A \otimes \psi \otimes \textrm{Id}_{X^{\otimes 2}})(\textrm{Id}_A \otimes \textrm{Id}_X \otimes \delta)\delta
\end{equation*}
(coassociativity),
\begin{equation*}\label{epsmor} 
m (\textrm{Id}_A \otimes \varepsilon) \psi = m (\varepsilon \otimes \textrm{Id}_A) 
\end{equation*}
($\varepsilon$ is a morphism in $\mathcal{T}_A^{\#}$),
\begin{equation*}\label{leftcou}
(m \otimes \textrm{Id}_X) (\textrm{Id}_A \otimes \varepsilon \otimes \textrm{Id}_X) \delta = u \otimes \textrm{Id}_X 
\end{equation*}
(left counit property),
\begin{equation*}\label{rightcou}
(m \otimes \textrm{Id}_X)(\textrm{Id}_A \otimes \psi)(\textrm{Id}_A \otimes \textrm{Id}_X \otimes \varepsilon) \delta = u \otimes \textrm{Id}_X 
\end{equation*}
 (right counit property).
\end{definition}

\subsection{Separable structures}

The notion of coseparable coalgebra was introduced by Larson in \cite{Lar} in terms of module injectivity. Bulacu, Caenepeel and Torrecillas proved an equivalent characterization of coseparability \cite[Prop.~7.3]{bct} that we choose as definition.
\begin{definition}\cite{bct}\label{eqcosep}
Let $(\mathcal{M}, \otimes, \underline{1})$ be a (strict) monoidal category. A coalgebra $(C, \Delta, \varepsilon)$ in $\mathcal{M}$ is called coseparable if there exists a morphism $B: C \otimes C \rightarrow \underline{1}$ in $\mathcal{M}$, such that
\begin{equation}\label{cosep}
(\textrm{Id}_C \otimes B)(\Delta \otimes \textrm{Id}_C)=(B \otimes \textrm{Id}_C)(\textrm{Id}_C \otimes \Delta) \quad \textrm{and} \quad B \Delta =\varepsilon.
\end{equation}
A morphism $B: C \otimes C \rightarrow \underline{1}$ satisfying the first of the two conditions \eqref{cosep} is called a Casimir morphism for $C$. A Casimir morphism satisfying also $B \circ \Delta= \varepsilon$ is called a normalized Casimir morphism. 
\end{definition}
Prompted by an example related to the tensor algebra, Ardizzoni and Menini introduced in \cite{am} a stronger version of the notion of separability for functors, called heavy separability. It was later shown in \cite{mt} that the correct corresponding notion for coalgebras is the following.
\begin{definition}{\cite[Def.~$2.2$]{mt}}\label{heavycosep}
Let $(C, \Delta, \varepsilon)$ be a coalgebra in a monoidal category $\mathcal{M}$. We will say $C$ is heavily coseparable (h-coseparable for short) if $C$ is coseparable, endowed with a normalized Casimir morphism $B$ that also satisfies 
\begin{equation*}\label{hcosep}
(B \otimes B)(\textrm{Id}_C \otimes \Delta \otimes \textrm{Id}_C) = B (\textrm{Id}_C \otimes \varepsilon \otimes \textrm{Id}_C) \quad  \textrm{(h-coseparability condition)}.
\end{equation*}
\end{definition}
Cowreaths are coalgebras in the monoidal category $\mathcal{T}_A^{\#}$, therefore one can accordingly define (h-)separable cowreaths.
\begin{definition}\cite[p. 240]{bct2}\label{sepcow}
A cowreath $\left(A,X,\psi \right)$ is called separable if $\left( X,\psi \right) $ is a coseparable coalgebra in the monoidal category $\mathcal{T}_{A}^{\#}$.
Similarly we will say that a cowreath $\left( A,X,\psi \right) $ is heavily separable (h-separable for short) if $\left( X,\psi \right) $ is a heavily coseparable coalgebra in the monoidal category $\mathcal{T}
_{A}^{\#}$.
\end{definition}
The theory of separable cowreaths in monoidal categories was developed in \cite{mt}.

\subsection{Separable cowreaths on two-sided Hopf modules}

From now on we will consider $\mathcal{M}=\mathrm{Vec}_{\Bbbk}$ the monoidal category of vector spaces over a fixed field $\Bbbk$. Remember that a bialgebra in $\mathrm{Vec}_{\Bbbk}$ is a $5$-tuple $(B, m, u, \Delta, \varepsilon)$, where $(B, m, u)$ is an algebra in $\mathrm{Vec}_{\Bbbk}$, $(B, \Delta, \varepsilon)$ is a coalgebra in $\mathrm{Vec}_{\Bbbk}$ and such that the following conditions hold
\[
\begin{array}{r r}
(m \otimes m)(\textrm{Id}_B \otimes \tau_{B,B} \otimes \textrm{Id}_B)(\Delta \otimes \Delta) =\Delta m, &  \Delta  u = u \otimes u,\\
\varepsilon  m = m (\varepsilon \otimes \varepsilon), & \varepsilon u = \mathrm{Id}_{\Bbbk},
\end{array}
\]
where $\tau_{B,B}: B \otimes B \rightarrow B \otimes B$ denotes the usual flip.
A Hopf algebra in $\mathrm{Vec}_{\Bbbk}$ is a $6$-tuple $(H, m, u, \Delta, \varepsilon, S)$, where $(H, m, u, \Delta, \varepsilon)$ is a bialgebra in $\mathrm{Vec}_{\Bbbk}$ and $S:H \rightarrow H$ is an inverse for $\mathrm{Id}_H$ in the convolution algebra $\mathrm{Hom}(H,H)$. The map $S$ is unique for every Hopf algebra and it is called the antipode of $H$.

Let $(A, \cdot, 1_A)$ be an algebra in $\mathrm{Vec}_{\Bbbk}$ and let $(H, m, u, \Delta, \varepsilon, S)$ be a Hopf algebra in $\mathrm{Vec}_{\Bbbk}$. Given $\rho_A: A \rightarrow A \otimes H$, the pair $(A, \rho_A)$ is called a (right) $H$-comodule algebra if the following conditions are satisfied.
\begin{align*}
(A, \rho_A) &  \textrm{ is a (right)} \ H\textrm{-comodule,}\\
\rho_A (ab) &= a_0b_0 \otimes a_1b_1,\\
\rho_A (1_A)&=1_A \otimes 1_H \ \ (\textrm{where } 1_H=u(1_{\Bbbk})).
\end{align*}
for all $a,b \in A$.
The map $\rho_A$ is also called a (right) $H$-coaction on $A$.

The following proposition shows how to construct a cowreath in $\mathrm{Vec}_{\Bbbk}$ from a Hopf algebra $H$ and an $H$-comodule algebra $A$ (see \cite{mt}). 
\begin{proposition}\cite[Sec. 5]{mt}\label{propXA}
Let $\left( H,m, u,\Delta, \varepsilon, S\right) $ be a Hopf algebra and $\left(A,\rho \right)$ be a right $H$-comodule algebra.
Let $\psi: H \otimes A \otimes H^{op} \rightarrow A \otimes H^{op} \otimes H$ be defined by
\begin{equation*}\label{psiHAHop}
\psi(h \otimes a \otimes l)=a_0 \otimes l_1 \otimes l_2 h a_1
\end{equation*}
for every $h \in H$ and every $a \otimes l \in A \otimes H^{op}$.
Then $(H, \psi) \in \mathcal{T}_{A\otimes H^{op}}^{\#}$ and $(H, \psi)$ is a coalgebra in $\mathcal{T}_{A \otimes H^{op}}^{\#}$, with comultiplication $\delta_H:  H \rightarrow A \otimes H^{op} \otimes H \otimes H$ and counit $\epsilon_H: H \rightarrow A \otimes H^{op}$ given by
\begin{eqnarray*}
\delta_H(h)=1_A \otimes 1_H\otimes h_1 \otimes h_2, \label{deltaH}\\
\epsilon_H(h)=\varepsilon_H(h)1_A \otimes 1_H, \label{epsH}
\end{eqnarray*}
for every $h \in H$.
\end{proposition}
Now, by explicitly writing the conditions in Definition~\ref{sepcow}, one can determine when such a cowreath is (h-)separable. 
\begin{theorem}\cite[Thm.~5.1]{mt}
\label{Theorem:Hbim}A cowreath $(A\otimes H^{op},H,\psi )$ is separable via
a Casimir element
\begin{equation*}
B:H\otimes H\rightarrow A\otimes H^{op}
\end{equation*}
given by
\begin{equation*}
h\otimes h^{\prime }\mapsto B^{A}\left( h\otimes h^{\prime }\right)
\otimes B^{H}\left( h\otimes h^{\prime }\right)
\end{equation*}%
if, and only if, $B$ satisfies the following conditions
\begin{align}
&a_{0}B^{A}(g_{2}ha_{1}\otimes g_{3}h^{\prime }a_{2})\otimes
B^{H}(g_{2}ha_{1}\otimes g_{3}h^{\prime }a_{2})g_{1}  \label{B1-4.23} \\
&=B^{A}(h\otimes h^{\prime })a\otimes gB^{H}(h\otimes h^{\prime }),  \notag
\end{align}
\begin{align}
&B^{A}(h_{2}\otimes h^{\prime })_{0}\otimes B^{H}(h_{2}\otimes h^{\prime
})_{1}\otimes B^{H}(h_{2}\otimes h^{\prime })_{2}h_{1}B^{A}(h_{2}\otimes
h^{\prime })_{1}  \label{B2-4.24} \\
&=B^{A}(h\otimes h_{1}^{\prime })\otimes B^{H}(h\otimes h_{1}^{\prime
})\otimes h_{2}^{\prime },  \notag
\end{align}%
\begin{equation}
B^{A}(h_{1}\otimes h_{2})\otimes B^{H}(h_{1}\otimes h_{2})=\varepsilon
(h)1_{A}\otimes 1_{H}.  \label{B3-4.25}
\end{equation}%
Moreover it is h-separable if and only if $B$ satisfies the further
condition%
\begin{align}
&B^{A}\left( h\otimes h_{1}^{\prime }\right) B^{A}\left( h_{2}^{\prime
}\otimes h^{\prime \prime }\right) \otimes B^{H} \left( h_{2}^{\prime }\otimes h^{\prime \prime}\right) B^{H}\left( h\otimes h_{1}^{\prime }\right)  \label{B4-4.26}\\
&=\varepsilon_H \left( h^{\prime }\right) B^{A}\left( h\otimes h^{\prime \prime }\right) \otimes B^{H}\left( h\otimes h^{\prime \prime}\right). \notag
\end{align}
\end{theorem}
Asking that the Casimir morphism $B$ be of a particular form results in easier conditions to check.
\begin{proposition}\cite[Prop.~5.2]{mt}
\label{PropReferee}A cowreath $(A\otimes H^{op},H,\psi )$ is separable via a
Casimir element%
\begin{equation*}
B:H\otimes H\rightarrow A\otimes H^{op}
\end{equation*}%
of the form
\begin{equation*}
h\otimes h^{\prime }\mapsto B^{A}\left( h\otimes h^{\prime }\right)
\otimes 1_{H}
\end{equation*}%
if and only if $B^{A}$ satisfies the following conditions
\begin{equation}
a_{0}B^{A}(ha_{1}\otimes h^{\prime }a_{2})=B^{A}(h\otimes h^{\prime })a,
\label{B1S}
\end{equation}%
\begin{equation}
B^{A}(1\otimes h)_{0}\otimes B^{A}(1\otimes h)_{1}=B^{A}(1\otimes
h_{1})\otimes h_{2},  \label{B2S}
\end{equation}%
\begin{equation}
B^{A}(h_{1}h^{\prime }\otimes h_{2}h^{\prime \prime })=\varepsilon_H
(h)B^{A}(h^{\prime }\otimes h^{\prime \prime })\text{ and }B^{A}\left(
1_{H}\otimes 1_{H}\right) =1_{A}.  \label{B3S}
\end{equation}%
Moreover, whenever $S$ is invertible, it is h-separable if and only if $B$
satisfies the further condition%
\begin{equation}
B^{A}\left( h\otimes 1_{H}\right) \cdot B^{A}\left( 1_{H}\otimes h^{\prime
}\right) =B^{A}\left( h\otimes h^{\prime }\right).  \label{B4S}
\end{equation}
\end{proposition}

\begin{remark}\label{sep2}
Proposition~\ref{PropReferee} was proved in \cite{mt} in order to construct examples of separable cowreaths which are not $h$-separable or not Frobenius (see Remarks $6.2$ ibid.).
Notice that a cowreath $(A \otimes H^{op},H, \psi)$ for which there is no map $B:H \otimes H \rightarrow A \otimes H^{op}$ of the form $h \otimes h' \mapsto B^{A}(h \otimes h') \otimes 1_H$ satisfying \eqref{B1S}-\eqref{B3S}, is not necessarily non-separable. Separability for such a cowreath could be attained via a morphism of a much more complicated form. The conditions under which an example of cowreath $(A \otimes H^{op},H, \psi)$ is separable via a Casimir morphism of a general form are described in \cite{mt2}.
\end{remark}

In view of the previous remark, we introduce the following definition to make a clear distinction.
\begin{definition}\label{rtsep}
A cowreath $(A \otimes H^{op},H, \psi)$ for which there is a map $B:H \otimes H \rightarrow A \otimes H^{op}$ of the form $h \otimes h' \mapsto B^{A}(h \otimes h') \otimes 1_H$ satisfying \eqref{B1S}-\eqref{B3S}, will be called right-trivially separable (rt-separable, for short).
Similarly, a cowreath $(A \otimes H^{op},H, \psi)$ for which there is a map $B:H \otimes H \rightarrow A \otimes H^{op}$ of the form $h \otimes h' \mapsto B^{A}(h \otimes h') \otimes 1_H$ satisfying \eqref{B1S}-\eqref{B4S} will be called right-trivially $h$-separable (rth-separable, for short).
\end{definition}

\begin{remark}
As observed in Remark~\ref{sep2}, a cowreath can be (h-)separable even if it is not rt(h)-separable. Clearly, a rt(h)-separable cowreath is (h-)separable.
\end{remark}

Our final goal is to determine a family of rt-separable cowreaths built with a $2^{n+1}$-dimensional Clifford algebra $A=Cl(\alpha,\beta_i, \gamma_i, \lambda_{ij})$ and its comodule algebra structure over the Hopf algebra $E(n)$, for every natural number $n$. We will show that in this case rt(h)-separability can be characterized by a family of equalities involving a tuple of maps related to the comodule algebra structure of $A$. We are not sure if the same approach could be used to characterize plain separability (i.e. conditions \eqref{B1-4.23}-\eqref{B4-4.26}), but it should be, at least in principle, a much harder task.

\subsection{The Hopf algebras \texorpdfstring{$E(n)$}{E(n)}}

From now on we will always assume $\chara \Bbbk \neq 2$.
The Hopf algebras denoted by $E(n)$ are a family that generalizes Sweedler's Hopf algebra $H_4=E(1)$ and that were introduced in \cite[p. $755$]{BDG} and studied in \cite{CD, CC, PVO, PVO2}. 
\begin{definition}{\cite[p.$18$]{CD}}\label{defEn} We denote by $E(n)$ the $2^{n+1}$-dimensional Hopf algebra over $\Bbbk$, generated by elements $g, x_1, \ldots, x_n$ such that $g^2=1$, $x_i^2=0$, $gx_i=-x_ig$ for any $i=1, \ldots, n$ and $x_ix_j=-x_jx_i$ for $i,j=1, \ldots, n$, $i <j$.
\end{definition}
The canonical Hopf algebra structure of $E(n)$ is given by
\[
\def\arraystretch{1.5}
\begin{array}{l l l l}
\Delta(g)=g \otimes g, & \varepsilon(g)=1, & S(g)=g^{-1}=g, &\\
\Delta(x_i)=x_i \otimes g +1 \otimes x_i, & \varepsilon(x_i)=0, & S(x_i)=gx_i, &  i=1, \ldots, n.
\end{array}
\]
For $P=\lbrace i_1,i_2,\ldots,i_s \rbrace \subseteq \lbrace 1,2, \ldots,n \rbrace$ such that $i_1 < i_2 < \ldots < i_s$, we denote $x_P = x_{i_1}x_{i_2}\cdots x_{i_s}$. If $P=\emptyset$ then $x_{\emptyset} = 1$. The set $\lbrace g^jx_P \ | \ P\subseteq \lbrace 1, \ldots ,n \rbrace, j \in \lbrace 0,1 \rbrace \rbrace$ is a basis of $E(n)$.
Let $F=\lbrace i_{j_1}, i_{j_2}, \ldots, i_{j_r} \rbrace$ be a subset of $P$ and define
\[S(F,P):=(j_1+\cdots+j_r)-\frac{r(r+1)}{2} \quad \textrm{and} \quad S(\emptyset,P):=0.\]
Then computations show that
\begin{equation}\label{deltagx}
\Delta(g^jx_P)=\sum_{F \subseteq P} (-1)^{S(F,P)} g^j x_F \otimes g^{|F|+j} x_{P \setminus F},
\end{equation}
\begin{equation}\label{antipode}
S(g^jx_P) = (-1)^{j|P|}g^{j+|P|}x_P,
\end{equation}
\begin{equation}\label{reorder}
(-1)^{S(F,P)}x_Fx_{P \setminus F}=x_P.
\end{equation}

In \cite{FR2} a characterization of $E(n)$-comodule algebra structures for finite-dimensional algebras was proved.
\begin{theorem}\cite[Thm. 4.4]{FR2}\footnote{The statement in \cite{FR2} replaces $(-1)^{\frac{|P|(|P|+1)}{2}}$ with $(-1)^{\left\lfloor \frac{|P|+1}{2}\right\rfloor}$.}\label{maingen}
Let $A$ be a finite dimensional algebra over a field $\Bbbk$ of characteristic $\chara \Bbbk \neq 2$. Then an $E(n)$-comodule algebra structure on $A$ is given by:
\begin{equation}\label{coactionfinal}
\rho(a)= \sum_{j=0,1}\underset{P \subseteq \lbrace 1, \ldots, n \rbrace}{\sum}(-1)^{\frac{|P|(|P|+1)}{2}} \varphi^j(d_P(a)) \otimes \frac{x_P +(-1)^{|P|+j}gx_P}{2},
\end{equation}
where
\begin{enumerate}
\item $\varphi$ is an automorphism  of $A$ of order $o(\varphi) \leq 2$ (i.e. an involution of $A$),
\item $d_{\emptyset}=\textrm{Id}_A$ and $d_P=d_{i_1}d_{i_2}\cdots d_{i_{|P|}}$ is a composition of $\varphi$-derivations such that $d_i^2 \equiv 0$, $\varphi d_i = -d_i \varphi$ and $d_id_j=-d_jd_i$.
\end{enumerate}
\end{theorem}
For sake of completeness we recall the definitions of involution and skew-derivation.
\begin{definition}
Let $A$ be an algebra and $\varphi:A \rightarrow A$ an algebra endomorphism. We will call $\varphi$ an algebra involution (or simply an involution) if $\varphi^2=\textrm{Id}_A$. 

\medskip

Fix an algebra map $\varphi:A \rightarrow A$. A $\Bbbk$-linear map $d:A \rightarrow A$ such that
\[d(ab)=d(a)b+\varphi(a)d(b)\]
for every $a, b \in A$ is called a $\varphi$-derivation (or a skew-derivation).
\end{definition}

\subsection{Clifford algebras}

The algebras $E(n)$ are a particular subfamily of well-known Clifford algebras. A Clifford algebra $Cl(V,q)$ is usually defined as a quotient of the tensor algebra $T(V)$ built on a free vector space $V$ by an ideal whose generators depend on a quadratic form $q$ \cite{C,La}. We prefer to make use of an equivalent characterization.
\begin{definition}\cite[Def.~$1$]{PVO2}\label{clt}
Let $\alpha, \beta_i,\gamma_i \in \Bbbk$ for $i=1, \ldots, n$ and $\lambda_{ij} \in \Bbbk$ for $i, j \in \lbrace 1, \ldots, n \rbrace$, $i<j$.
The Clifford algebra $Cl(\alpha,\beta_i,\gamma_i, \lambda_{ij})$ is the unital associative algebra generated by elements $G, X_1, \ldots, X_n$ such that $G^2=\alpha$, $X_i^2=\beta_i$, $GX_i+X_iG=\gamma_i$ for all $i=1, \ldots, n$ and $X_iX_j+X_jX_i=\lambda_{ij}$ for all $i,j \in \lbrace 1, \ldots, n \rbrace$ with $i<j$. A $\Bbbk$-basis for this algebra is given by $\lbrace G^j X_{P} \rbrace$, where $j=0,1$ and $X_P=X_{i_1} \cdots X_{i_s}$ with $P=\lbrace i_1 < i_2 < \ldots < i_s \rbrace \subseteq \lbrace 1, \ldots, n \rbrace$. 
\end{definition}
\begin{notation}
We warn the reader about the slightly misleading notation $Cl(\alpha,\beta_i,\gamma_i, \lambda_{ij})$. This should always be read as
\[Cl(\alpha, \beta_1, \beta_2, \ldots, \beta_n, \gamma_1, \gamma_2, \ldots, \gamma_n, \lambda_{11}, \lambda_{12}, \ldots, \lambda_{1n}, \lambda_{21}, \ldots, \lambda_{nn})\]
and considered a single algebra (and not a family) once all the defining scalars are fixed.
\end{notation}

\begin{remark}
Definition~\ref{clt} appears for the first time in \cite{PVO2}, where the authors name these objects ``Clifford-type'' algebras. As shown in \cite[Thm.~3.2]{FR2}, if $\Bbbk$ is a field of characteristic $\chara \Bbbk \neq 2$, Clifford-type algebras over $\Bbbk$ are (isomorphic to) classical Clifford algebras defined as quotients of tensor algebras, hence we omit the suffix ``-type''.
\end{remark}
In \cite{PVO2} it was proved that a $2^{n+1}$-dimensional Clifford algebra $A=Cl(\alpha, \beta_i, \gamma_i, \lambda_{ij})$ admits a canonical $E(n)$-comodule algebra structure $\rho: A \rightarrow A \otimes E(n)$ that makes it a cleft extension of the ground field $\Bbbk$. This coaction is given by
\begin{align}
\rho(1_A)&=1_A \otimes 1_{E(n)}\label{canonical1}\\
\rho(G) &=G \otimes g \label{canonicalG}\\
\rho(X_i) &= X_i \otimes g + 1_A \otimes x_i, \quad i=1, \ldots, n\label{canonicalXi} \\
\rho(GX_i) &= GX_i \otimes 1_{E(n)} +  G \otimes gx_i, \quad i=1, \ldots, n. \label{canonicalGXi}
\end{align}
\begin{proposition}\label{canonictuple}
Let $A=Cl(\alpha, \beta_i, \gamma_i, \lambda_{ij})$ be a Clifford algebra. The tuple of maps corresponding to its canonical $E(n)$-comodule algebra structure is $(\sigma, d_1, \ldots, d_n)$, where $\sigma$ is the main involution of $A$ and each $d_j$ is determined by $d_j(G)=0$, $d_j(X_i)=-\delta_{ij}$ for every $i=1,\ldots, n$. 
\end{proposition}
\begin{proof}
By Theorem~\ref{maingen} we know that a $E(n)$ coaction $\rho:A \rightarrow A \otimes E(n)$ on $A$ can be expressed explicitly as
\begin{align*}
\rho(a)&=\frac{a+\varphi(a)}{2} \otimes 1+ \frac{a-\varphi(a)}{2} \otimes g-\sum_{i=1}^n\left[ d_i \left(\frac{a-\varphi(a)}{2} \right) \otimes x_i-d_i \left(\frac{a+\varphi(a)}{2}\right) \otimes gx_i \right]+\\
&+\underset{|P|>1}{\sum}(-1)^{\frac{|P|(|P|+1)}{2}}\left[  d_P(a) \otimes \frac{x_P +(-1)^{|P|}gx_P}{2} + \varphi(d_P(a)) \otimes \frac{x_P +(-1)^{|P|+1}gx_P}{2} \right]
\end{align*}
for all $a \in A$, where $(\varphi, d_1, \ldots, d_n)$ is a $n+1$-uple of suitable maps. By definition, $\rho$ is the canonical $E(n)$-comodule structure of $A$ if, and only if, equalities \eqref{canonical1}-\eqref{canonicalGXi} are satisfied and one can easily check that this is true if, and only if, the maps $(\varphi, d_1, \ldots, d_n)$ satisfy
\begin{align*}
\varphi(G)&=-G, \\
\varphi(X_i)&=-X_i \quad \textrm{for every } i=1, \ldots, n,\\
d_j(G)&=0 \qquad \textrm{for every } j=1, \ldots, n,\\
d_j(X_i)&=-\delta_{ij} \qquad \textrm{for every } i,j=1, \ldots, n.
\end{align*}
Since $\lbrace G, X_1, \ldots, X_n \rbrace$ is a basis for the vector space $V$ underlying the Clifford algebra $A$, we see that $\varphi|_{A_0}=\textrm{Id}_{A_0}$ and $\varphi|_{A_1}=-\textrm{Id}_{A_1}$, therefore $\varphi=\sigma$, the main involution of the Clifford algebra $A$ (see e.g. \cite[Def. IV.3.7]{La}).
\end{proof}
\begin{notation}
The choice of a distinguished generator $G$ follows the notation established in \cite{PVO2}. It comes from the fact that these algebras can be regarded as $E(n)$-cleft extensions over the ground field $\Bbbk$. The generators of the Hopf algebra $E(n)$ are usually called $g$, $x_1$, \ldots, $x_n$ \--- following a quite widespread notation for grouplike elements and skew-primitives elements \--- whence the choice of letters for the generators of $Cl(\alpha,\beta_i,\gamma_i, \lambda_{ij})$. The reader interested in the exact correspendence between these two families of generators is referred to \cite{PVO2}. We will just point out that each Hopf algebra $E(n)$ can be regarded as a Clifford-type algebra itself, namely $E(n)=Cl(1,0,0,0)$. 
Nonetheless, in some cases, it will be convenient to treat $G$ exactly as one of the remaining generators. To this aim we introduce the notation
\[X_0:=G, \quad \beta_0:=\alpha, \quad \textrm{and} \quad \lambda_{0i}:=\gamma_i \quad \textrm{for every } i=1, \ldots, n.\]
In this way we can still write $X_R=X_{i_1}\cdots X_{i_s}\subseteq \lbrace 0,1, \ldots, n \rbrace$ with $R=\lbrace i_1 < \ldots <i_s \rbrace$ to indicate an element $G^jX_P$ of the canonical basis. Clearly $j \neq 0 \iff i_1 =0$. Notice that with this notation, \eqref{reorder} holds for every $F \subseteq P \subseteq \lbrace 0,1, \ldots, n \rbrace$.
\end{notation}

\section{Technical results}\label{technicalresults}

In this section we prove a series of simple, yet crucial results that will help us with calculations in the sequel. 

\subsection{Perfect matchings}
Perfect matchings are usually defined in the contest of graphs, but they can also be useful when working with partitions of a set. We will use them to explicitly describe the Casimir element realizing the separability of each cowreath in our main result.

Given a graph $\mathcal{G}=(V,E)$ with vertices $V$ and edges $E$, a perfect matching in $\mathcal{G}$ is a subset $M$ of $E$ such that every vertex in $V$ is adiacent to 
exactly one edge in $M$ (an equivalent notion is that of a 1-factor of $\mathcal{G}$). Notice that there is a perfect matching in $\mathcal{G}$
if, and only if, $|V|=2n$. 
\begin{figure}[H]
\centering
\begin{tikzpicture}
        \filldraw [black](1,0) circle [radius=2pt]
						(1,1) circle [radius=2pt]
						(2,1) circle [radius=2pt]
						(2,0) circle [radius=2pt];	
		\coordinate [label=below:$3$] (A) at (1,0);
		\coordinate [label=above:$1$] (B) at (1,1);
		\coordinate [label=above:$2$] (C) at (2,1);
		\coordinate [label=below:$4$] (D) at (2,0);
		\draw (A)--(D) (B)--(C); 	

        \filldraw [black](5,0) circle [radius=2pt]
						(5,1) circle [radius=2pt]
						(6,1) circle [radius=2pt]
						(6,0) circle [radius=2pt];	
		\coordinate [label=below:$3$] (E) at (5,0);
		\coordinate [label=above:$1$] (F) at (5,1);
		\coordinate [label=above:$2$] (G) at (6,1);
		\coordinate [label=below:$4$] (H) at (6,0);
		\draw (E)--(F) (G)--(H); 	

        \filldraw [black](9,0) circle [radius=2pt]
						(9,1) circle [radius=2pt]
						(10,1) circle [radius=2pt]
						(10,0) circle [radius=2pt];	
		\coordinate [label=below:$3$] (L) at (9,0);
		\coordinate [label=above:$1$] (M) at (9,1);
		\coordinate [label=above:$2$] (N) at (10,1);
		\coordinate [label=below:$4$] (O) at (10,0);
		\draw (L)--(N) (M)--(O); 	
\end{tikzpicture}
\caption{Perfect matchings in the complete graph $K_4$.}
\label{2edges}
\end{figure}
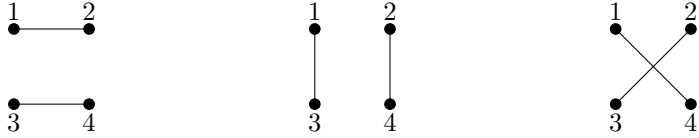

\begin{remark}
If $\mathcal{G}=K_{2n}$, the complete graph on $n$ vertices, then $\mathcal{G}$ has $(2n-1)!!$ perfect matchings. 
\end{remark}

\begin{proof}
A perfect matching in $K_{2n}$ is obtained by choosing pairs of vertices from $V$ without any repetition and until every vertex is picked.
The total occurrences are
\[\binom{2n}{2}\binom{2n-2}{2}\binom{2n-4}{2} \cdots \binom{2}{2}=\frac{(2n)!}{2^{n}},\]
if we keep track of the order in which we choose each pair. Since each perfect matching is determined by a partition of the vertices and not by the order in which we pick them, we need to identify cases that are just permutations of each other. Therefore we get that the number of perfect matchings in $K_{2n}$ is
\[\frac{(2n)!}{2^n n!}=(2n-1)!!.\]
\end{proof}
It is clear that we can extend the concept of perfect matching to any set having an even number of elements.
\begin{definition}\label{defpm}
Let $S=\lbrace s_1, s_2, s_3, \ldots, s_{2n} \rbrace$ be a non-empty set with cardinality $2n$. Then we define a perfect matching for $S$ as a partition of $S$ into blocks of cardinality $2$. If $S=\emptyset$ we define $\emptyset$ as its only perfect matching. We also denote with $\pi_2(S)$ the set of all perfect matchings of $S$.
\end{definition}
For example if $S=\lbrace s_1, s_2, s_3, s_4, s_5, s_6 \rbrace$ we have that 
\[\mathcal{P}_1=\lbrace \lbrace s_1,s_2 \rbrace, \lbrace s_3,s_4 \rbrace, \lbrace s_5,s_6 \rbrace \rbrace \textrm{ and } \mathcal{P}_2=\lbrace \lbrace s_1,s_4 \rbrace, \lbrace s_2,s_6 \rbrace, \lbrace s_3,s_5 \rbrace \rbrace\]
are perfect matchings for $S$. Notice that $|\pi_2(S)|=(|S|-1)!!$, since the perfect matchings for a non-empty set $S$ correspond exactly to those in the complete graph $K_{|S|}$, and $(-1)!!=1$.
\begin{proposition}\label{prop:pmrecur}
Let $S=\lbrace s_1, s_2, \ldots, s_{2n} \rbrace$ be a set of cardinality $|S|=2n \geq 2$. Pick an element $\overline{s} \in S$. Then we have
\[\pi_2(S)=\left \lbrace \mathcal{P} \sqcup \lbrace \lbrace i, \overline{s} \rbrace \rbrace \ | \ i \in S \setminus \lbrace \overline{s} \rbrace, \ \mathcal{P} \in \pi_2(S \setminus \lbrace i, \overline{s} \rbrace) \right \rbrace.\]
\end{proposition}
\begin{proof}
Let 
\[T:=\left \lbrace \mathcal{P} \sqcup \lbrace \lbrace i, \overline{s} \rbrace \rbrace \ | \ i \in S \setminus \lbrace \overline{s} \rbrace, \ \mathcal{P} \in \pi_2(S \setminus \lbrace i, \overline{s} \rbrace) \right \rbrace.\]
Observe that for every $i \in S \setminus \lbrace \overline{s} \rbrace$ we have $|\pi_2(S \setminus \lbrace i, \overline{s} \rbrace)|=(|S|-3)!!$. On the other hand $|S \setminus \lbrace \overline{s} \rbrace|=|S|-1$, therefore
\[|T|=(|S|-1) \cdot (|S|-3)!!=(|S|-1)!!.\] Since $T$ and $\pi_2(S)$ share the same cardinality, it is enough to show that $T \subseteq \pi_2(S)$ to conclude that they coincide. Let us consider an element $\mathcal{P}' \in T$. By definition we have $\mathcal{P'}=\mathcal{P} \sqcup \lbrace \lbrace i, \overline{s} \rbrace \rbrace$ for some $i \in S \setminus \lbrace \overline{s} \rbrace$ and some $\mathcal{P} \in \pi_2(S \setminus \lbrace i, \overline{s} \rbrace)$. Every element in $\mathcal{P}'$ is non-empty since $\mathcal{P}$ is a partition (or either the empty set if $|S|=2$) and $\lbrace i, \overline{s} \rbrace$ is non-empty. Furthermore, the union of the elements in $\mathcal{P}'$ is $S$, since the union of those in $\mathcal{P}$ give $S \setminus \lbrace i, \overline{s} \rbrace$ by definition. Finally all the elements in $\mathcal{P}$ are pairwise disjoint (since $\mathcal{P}$ is a partition of $S \setminus \lbrace i, \overline{s} \rbrace$) and have also empty intersection with $\lbrace i, \overline{s} \rbrace$. Clearly, all elements of $\mathcal{P}'$ are subsets of $S$ with cardinality $2$ and therefore $\mathcal{P}' \in \pi_2(S)$.
\end{proof}

When $S=\lbrace s_1, s_2, s_3, \ldots, s_{2n} \rbrace \subseteq \mathbb{Z}$, $S \neq \emptyset$, we can always assume that its elements are labeled in increasing order, i.e. that $s_1<s_2<s_3<\ldots<s_{2n-1}<s_{2n}$. Then we can fix a canonical form for every perfect matching $\mathcal{P}$ for $S$. Indeed a perfect matching $\mathcal{P}$ can be regarded as a collection of $n$ blocks, which we label in increasing order in the following sense: $\mathcal{P}=\lbrace P_1, P_2, \ldots, P_n \rbrace$, where $P_k=\lbrace s_{k_1} <s_{k_2} \rbrace$ and $s_{1_2}<s_{2_2}<\ldots<s_{n_2}$. We will write $\mathcal{P}=P_1|P_2|\ldots|P_n$ to display this ordered presentation of a perfect matching. For example, if $S=\lbrace 2,3,5,8 \rbrace$, we have
\[\pi_2(S)=\lbrace  \lbrace 2,3 \rbrace|\lbrace 4,8 \rbrace ,  \lbrace 2,5 \rbrace|\lbrace 3,8 \rbrace,  \lbrace 3,5 \rbrace|\lbrace 2,8 \rbrace \rbrace .\]
Moreover, when $S=\lbrace s_1<s_2<\ldots<s_{2n} \rbrace \subseteq \mathbb{Z}$, a perfect matching $\mathcal{P}=P_1|P_2| \ldots|P_n \in \pi_2(S)$ is in correspondence with a permutation $\sigma_{\mathcal{P}}$ of the set $S$ via the following assignment
\begin{equation}\label{matchperm}
\mathcal{P}=P_1|P_2| \ldots|P_n \longmapsto \sigma_{\mathcal{P}}=\begin{pmatrix}
s_1 & s_2 & s_3 & s_4 & \ldots & s_{2n-1} & s_{2n}\\
s_{1_1} & s_{1_2} & s_{2_1} & s_{2_2} & \ldots & s_{n_1} & s_{n_2}
\end{pmatrix},
\end{equation}
i.e. $\sigma_{\mathcal{P}}(s_i)=s_{{\frac{i+1}{2}}_1}$ for every odd $i$ and $\sigma_{\mathcal{P}}(s_i)=s_{{\frac{i}{2}}_2}$ for every even $i$. This correspondence is well-defined (and injective) thanks to the fact that we have fixed a canonical form for each perfect matching by means of the total order on $\mathbb{Z}$. Another consequence of this fact is that we always have $s_{n_2}=s_{2n}$. Now, since each permutation has a sign, we can introduce the following definition.
\begin{definition}\label{sign}
Let $S=\lbrace s_1 < s_2 < \ldots < s_{2n} \rbrace \subseteq \mathbb{Z}$, $S \neq \emptyset$ and $\mathcal{P} \in \pi_2(S)$. We define the sign of the perfect matching $\mathcal{P}$ as the sign of the associated permutation $\sigma_{\mathcal{P}}$ described in \eqref{matchperm}:
\[\sgn\left(\mathcal{P}\right):=\sgn\left(\sigma_{\mathcal{P}}\right).\]
For $S=\emptyset$ we set $\sgn\left(\emptyset\right):=1$.
\end{definition}
\begin{proposition}\label{prop:matching}
Let $n \geq 1$ and $S=\lbrace s_1 < s_2 < \ldots < s_{2m} \rbrace \subseteq \lbrace 0,1, \ldots, n \rbrace$ with $|S|=2m \geq 0$. We adopt the notation of Definition~\ref{defEn}.
\begin{enumerate}[i)]
\item For every $\mathcal{P}=P_1|P_2|\ldots|P_m \in \pi_2(S)$
\[x_S=\sgn\left(\mathcal{P}\right)x_{P_1}x_{P_2}\ldots x_{P_m}.\]
\label{prop:bou}
\item \[\sgn\left(\mathcal{P}\right)=(-1)^{\sum_{k=1}^m S(P_k,P_k \sqcup \ldots \sqcup P_m)}.\] \label{cor:sgn}
\item \[\sum_{\mathcal{P} \in \pi_2(S)}\sgn\left(\mathcal{P}\right)=1.\]\label{prop:sgn}
\end{enumerate}
\end{proposition}

\begin{proof}
\begin{enumerate}[\it i)]
\item See \cite[Lemma 3, p. 473]{bou}.
\item \[\sgn\left(\mathcal{P}\right)x_{P_1}x_{P_2}\ldots x_{P_m}=x_S\overset{\eqref{reorder}}{=}(-1)^{\sum_{k=1}^m S(P_k,P_k \sqcup \ldots \sqcup P_m)}x_{P_1}x_{P_2}\ldots x_{P_m}.\]
\item The proof is by induction on $m$. If $m=0$, then $S=\emptyset$ and we have $\sum_{\mathcal{P} \in \pi_2(S)}\sgn\left(\mathcal{P}\right)=\sgn\left(\emptyset\right)=1$. If $m=1$, then we have $S=\lbrace s_1<s_2 \rbrace=P_1$ and $\mathcal{P}=\lbrace P_1 \rbrace$. Clearly $\sigma_{\mathcal{P}}=\textrm{Id}_S$ and thus $\sum_{\mathcal{P} \in \pi_2(S)}\sgn\left(\mathcal{P}\right)=\sgn\left(\textrm{Id}_S\right)=1$.

Now suppose our statement holds true for some natural $m \geq 1$ such that $2m \leq n$ and consider $T=\lbrace s_1<s_2<\ldots<s_{2m}<s_{2m+1}<s_{2m+2} \rbrace \subseteq \lbrace 0,1, \ldots, n \rbrace$. For every $\mathcal{P}' \in \pi_2(T)$ we can write $\mathcal{P}'=P_1|P_2|\ldots|P_m|P_{m+1}=\mathcal{P}|P_{m+1}$ where $P_{m+1}=\lbrace t<s_{2m+2} \rbrace$ for some $t \in T\setminus \lbrace s_{2m+2} \rbrace$ and some $\mathcal{P} \in \pi_2(T\setminus P_{m+1})$, thanks to Proposition~\ref{prop:pmrecur}.
We have
\begin{align*}
\sgn(\mathcal{P}')x_T\overset{\ref{prop:bou})}&{=}x_{P_1}x_{P_2}\ldots x_{P_m}x_{P_{m+1}}\overset{\ref{prop:bou})}{=}\sgn(\mathcal{P})x_{T \setminus P_{m+1}}x_{P_{m+1}}=\sgn(\mathcal{P})x_{P_{m+1}}x_{T \setminus P_{m+1}}\\\overset{\eqref{reorder}}&{=}\sgn(\mathcal{P})(-1)^{S(P_{m+1},T)}x_T=\sgn(\mathcal{P})(-1)^{S(\lbrace t \rbrace,T)}x_T,\\
\end{align*}
that is $\sgn(\mathcal{P}|P_{m+1})=\sgn(\mathcal{P}')=\sgn(\mathcal{P})(-1)^{S(\lbrace t \rbrace,T)}$.
Then
\begin{align*}
\sum_{\mathcal{P'} \in \pi_2(T)}\sgn\left(\mathcal{P}'\right)\overset{Prop.~\ref{prop:pmrecur}}&{=}\sum_{t \in T \setminus \lbrace s_{2m+2} \rbrace}\sum_{\mathcal{P} \in \pi_2(T \setminus P_{m+1})}\sgn\left(\mathcal{P}|P_{m+1}\right)\\
&=\sum_{t \in T \setminus \lbrace s_{2m+2} \rbrace}\sum_{\mathcal{P} \in \pi_2(T \setminus P_{m+1})}\sgn(\mathcal{P})(-1)^{S(\lbrace t \rbrace,T)}\\
\overset{Ind. \ hyp.}&{=}\sum_{t \in T \setminus \lbrace s_{2m+2} \rbrace}(-1)^{S(\lbrace t \rbrace,T)}=\sum_{k=1}^{2m+1}(-1)^{k-1}=\sum_{k=0}^{2m}(-1)^{k}=1.
\end{align*}
\end{enumerate}
\end{proof} 

\subsection{Permuting factors in Clifford algebras}

Clifford algebras are non-commutative. Nonetheless we will see that in these algebras the order of factors in a product is somewhat controlled by the function $S(F, P)$ introduced after Definition~\ref{defEn}, similarly to what happens for the Hopf algebras $E(n)$ (see \eqref{reorder}). The following lemma displays a list of properties satisfied by this function.
\begin{lemma}\label{lemmacrucial}
Let $F \subseteq P \subseteq \lbrace 1, \ldots, n \rbrace$.
\begin{enumerate}[i)]
\item For every $i \in P$, we have
\[S(P \setminus \lbrace i \rbrace,P)=|P|-1-S(\lbrace i \rbrace, P)=\textrm{ number of elements of } P \textrm{ that are bigger than } i.
\]\label{lemmaPi}
\item \[\sum_{l \in F} S(\lbrace l \rbrace, P)=S(F,P)+\frac{|F|(|F|-1)}{2}.\] \label{lemmaSF1}
\item \[S(F,P)+S(P \setminus F,P)=|F|(|P|-|F|).\]\label{corcompl}
\item For every $F' \subseteq P \setminus F$ we have
\[\sum_{j \in F'} S(\lbrace j \rbrace, F \sqcup \lbrace j \rbrace)=S(F',P)-S(F',P\setminus F).\] \label{LemmaSF2}
\item For every $F' \subseteq F$ we have
\begin{align*}
S(F \setminus F', P \setminus F')=S(F,P)-\sum_{l \in F'}\left[S(\lbrace l \rbrace, P)-S(\lbrace l \rbrace, F)\right]=S(F,P)-S(F',P)+S(F',F).
\end{align*}\label{lemmaFPF}
\item For every $i \in \lbrace 1, \ldots, n \rbrace$ such that $i>j$ for every $j \in P$ we have
\[S(F \sqcup \lbrace i \rbrace,P \sqcup \lbrace i \rbrace)=S(F, P)+|P|-|F|.\]\label{cor1}
\end{enumerate}
\end{lemma}
\begin{proof}
Let $P=\lbrace i_1 < i_2 < \ldots < i_{|P|} \rbrace$ for every occurence of $P$ throughout the proof.
\begin{enumerate}[\it i)]
\item Let $i=i_t$. Then
\[S(P \setminus \lbrace i \rbrace,P)=1+\ldots+|P|-t-\frac{|P|(|P|-1)}{2}=|P|-t=|P|-S(\lbrace i \rbrace,P)-1.\]
The second equality follows from $S(P \setminus \lbrace i \rbrace,P)=|P|-t$.
\item Let $F=\lbrace i_{j_1}, \ldots, i_{j_r}\rbrace$. Then
\[\sum_{l \in F} S(\lbrace l \rbrace, P)=\sum_{k=1}^r (j_k-1)=S(F,P)+\frac{r(r+1)}{2}-r=S(F,P)+\frac{r(r-1)}{2}.\]
\item \begin{align*}
S(F,P)+S(P \setminus F,P)\overset{\ref{lemmaSF1})}&{=}\sum_{i \in P} S(\lbrace i \rbrace, P)-\frac{|F|^2-|F|}{2}-\frac{|P \setminus F|(|P \setminus F|-1)}{2}\\
\overset{\ref{lemmaSF1})}&{=}S(P, P)+\frac{|P|^2-|P|}{2}+|F|(|P|-|F|)-\frac{|P|^2-|P|}{2}\\
&=|F|(|P|-|F|).
\end{align*}
\item Let $F'=\lbrace i_{j_1}, \ldots, i_{j_r}\rbrace$. By definition we have
$S(F',P)=j_1+j_2+\ldots+j_r-\frac{r(r+1)}{2}$.
Moreover
\[S(F',P \setminus F)=(j_1-n(j_1))+(j_2-n(j_2))+\ldots+(j_r-n(j_r))-\frac{r(r+1)}{2},\]
where $n(j_t)=$ number of elements in $F$ that are smaller than $i_{j_t}$.
Therefore
\begin{align*}
S(F',P)-S(F',P \setminus F)&=\sum_{t=1}^r n(j_t)\overset{\ref{lemmaPi})}{=}\sum_{t=1}^r S(\lbrace i_{j_t} \rbrace, F \sqcup \lbrace i_{j_t} \rbrace)=\sum_{l \in F'} S(\lbrace l \rbrace, F \sqcup \lbrace l \rbrace).
\end{align*}
\item The proof of the first part is by induction on $|F'|$. If $F'=\emptyset$ the statement is trivially true. Suppose $F'=\lbrace i \rbrace$. Then we must prove that
\[S(F \setminus \lbrace i \rbrace, P \setminus \lbrace i \rbrace)=S(F,P)-S(\lbrace i \rbrace, P)+S(\lbrace i \rbrace, F).\]
Remember that $P=\lbrace i_1 < i_2 < \ldots < i_{|P|} \rbrace$ and let $F=\lbrace i_{j_1}, i_{j_2}, \ldots, i_{j_r} \rbrace$, $i=i_{j_t}$.
Then every element $i_k \in P$ has the same index $k$ when regarded in $P \setminus \lbrace i \rbrace$ if $k<j_t$, but changes index to $k-1$ if $k>j_t$.
Therefore, since $S(F,P)=j_1+j_2+ \ldots +j_r-\frac{r(r+1)}{2}$,
then
\begin{align*}
S(F\setminus \lbrace i_{j_t} \rbrace, P \setminus \lbrace i_{j_t} \rbrace)&=j_1+\ldots+j_{t-1}+(j_{t+1}-1)+\ldots+(j_r-1)-\frac{r(r-1)}{2}\\
&=S(F,P)-j_t+t=S(F,P)-S(\lbrace i_{j_t} \rbrace,P)-1+t\\
&=S(F,P)-S(\lbrace i_{j_t} \rbrace,P)+S(\lbrace i_{j_t} \rbrace, F).
\end{align*}
Now suppose the statement holds for every $F' \subseteq F$ of a fixed cardinality $1 \leq m < |F|$ and consider a subset $F'' \subseteq F$ of cardinality $|F''|=m+1$. Let $F''=F' \sqcup \lbrace j'' \rbrace$ so that $|F'|=m$. We can assume, withouth losing of generality, that $j''>l$ for every $l \in F'$.
Then
\begin{align*}
S(F \setminus F'', P \setminus F'')&=S((F \setminus F')\setminus \lbrace j'' \rbrace, (P \setminus F') \setminus \lbrace j'' \rbrace)\\
\overset{Base \ case}&{=}S(F \setminus F',P \setminus F')-S(\lbrace j'' \rbrace, P \setminus F')+S(\lbrace j'' \rbrace, F \setminus F')\\
\overset{Ind. \ hyp.}&{=}S(F,P)-\sum_{l \in F'}\left[S(\lbrace l \rbrace, P)-S(\lbrace l \rbrace, F)\right]-S(\lbrace j'' \rbrace, P \setminus F')+\\
&+S(\lbrace j'' \rbrace, F \setminus F').
\end{align*}
Since $j''>l$ for every $l \in F'$, it follows that $S(\lbrace j'' \rbrace, P \setminus F')=S(\lbrace j'' \rbrace, P)-|F'|$ and $S(\lbrace j'' \rbrace, F \setminus F')=S(\lbrace j'' \rbrace, F)-|F'|$.
Thus 
\begin{align*}
S(F \setminus F'', P \setminus F'')&=S(F,P)-\sum_{l \in F'}\left[S(\lbrace l \rbrace, P)-S(\lbrace l \rbrace, F)\right]-S(\lbrace j'' \rbrace, P)+S(\lbrace j'' \rbrace, F)\\
&=S(F,P)-\sum_{l \in F''}\left[S(\lbrace l \rbrace, P)+S(\lbrace l \rbrace, F)\right].
\end{align*}
The second part of the equality follows by \textit{\ref{lemmaSF1})}.
\item Substitute $F$ with $F \sqcup \lbrace i \rbrace$, $P$ with $P \sqcup \lbrace i \rbrace$ and put $F'=\lbrace i \rbrace$ in \textit{\ref{lemmaFPF})}.
\end{enumerate}
\end{proof}

Let $A=Cl(\alpha,\beta_i,\gamma_i, \lambda_{ij})$ be a Clifford algebra. Remember that its generators are $G,X_1, \ldots, X_n$ and that a $\Bbbk$-basis for this algebra can be written as $\lbrace X_P \rbrace$ where  $X_P=X_{i_1} \cdots X_{i_s}$ with $P=\lbrace i_1 < i_2 < \ldots < i_s \rbrace \subseteq \lbrace 0,1, \ldots, n \rbrace$, with the convention that $X_0=G$. We also recall that $X_i^2=\beta_i$ and $X_iX_j+X_jX_i=\lambda_{ij}$ for every $i,j=1, \ldots, n$, $i < j$, with the convention that $\beta_0=\alpha$, $\lambda_{0i}=\gamma_i$ for every $i=1, \ldots, n$.

\begin{notation}
Some of the statements and proofs in the sequel may contain the following abuse of notation. We have defined $\lambda_{i_1i_2}$ only when $i_1<i_2$, but whenever a $\lambda_{j_1j_2}$ with $j_1>j_2$ is encountered it should always be interpreted as $\lambda_{j_2j_1}$.
\end{notation}
\begin{lemma} 
Let $P \subseteq \lbrace 0, 1, \ldots, n \rbrace$. We have
\item \begin{equation}\label{XRXj}
X_PX_j=\sum_{i \in P \setminus \lbrace j \rbrace} (-1)^{S(P \setminus \lbrace i,j \rbrace,P \setminus \lbrace j \rbrace)}\lambda_{ij} X_{P \setminus \lbrace i \rbrace} +(-1)^{|P|-\delta_{j \in P}}X_jX_P
\end{equation}
for every $j=0, \ldots, n$, and
\begin{equation}\label{betacoroll}
X_PX_j=\sum_{\underset{i \in P}{i>j}} (-1)^{S(P \setminus \lbrace i \rbrace,P)}\lambda_{ij} X_{P \setminus \lbrace i \rbrace} +(-1)^{S(P \setminus \lbrace j \rbrace,P)}\beta_j X_{P \setminus \lbrace j \rbrace}
\end{equation}
for every $j \in P$. 
\end{lemma}

\begin{proof}
We prove \eqref{XRXj} by induction on $|P|$. This equality is trivially true when $P=\emptyset$, whereas, if we take $P=\lbrace l \rbrace$, we get
\[X_lX_j=\sum_{i \in \lbrace l \rbrace \setminus \lbrace j \rbrace} (-1)^{S(\lbrace l \rbrace \setminus \lbrace i,j \rbrace,\lbrace l \rbrace \setminus \lbrace j \rbrace)}\lambda_{ij} X_{\lbrace l \rbrace \setminus \lbrace i \rbrace} +(-1)^{1-\delta_{jl}}X_jX_l.\]
Hence we find $X_lX_l=X_lX_l$ if $j=l$, and $X_lX_j=\lambda_{lj} -X_jX_l$ if $j \neq l$, which are both clearly true by definition. Now suppose \eqref{XRXj} holds true for every $P \subseteq \lbrace 0, 1, \ldots, n \rbrace$ with cardinality $|P|=m \geq 1$. Consider $P'=\lbrace i_1 < i_2 < \ldots <i_m < i_{m+1} \rbrace\subseteq \lbrace 0, 1, \ldots, n \rbrace$ with cardinality $m+1$ and split it as $P'=P \sqcup \lbrace i_{m+1} \rbrace$, where $P=\lbrace i_1 < i_2 < \ldots <i_m \rbrace$. Clearly $X_{P'}=X_PX_{i_{m+1}}$. 
If we take $j=i_{m+1}$, then $j \notin P$ and
\begin{align*}
X_{P'}X_j&=X_PX_jX_j\overset{Ind. \ hyp.}{=}\left( \sum_{i \in P} (-1)^{S(P \setminus \lbrace i \rbrace,P)}\lambda_{ij} X_{P \setminus \lbrace i \rbrace} +(-1)^{|P|}X_jX_P\right)X_j\\
&= \sum_{i \in P' \setminus \lbrace j \rbrace } (-1)^{S(P' \setminus \lbrace i,j \rbrace,P' \setminus \lbrace j \rbrace)}\lambda_{ij} X_{P' \setminus \lbrace i \rbrace} +(-1)^{|P'|-1}X_jX_{P'}\\
&= \sum_{i \in P' \setminus \lbrace j \rbrace } (-1)^{S(P' \setminus \lbrace i,j \rbrace,P' \setminus \lbrace j \rbrace)}\lambda_{ij} X_{P' \setminus \lbrace i \rbrace} +(-1)^{|P'|-\delta_{j \in P'}}X_jX_{P'}.\\
\end{align*}
On the other hand, if $j \neq i_{m+1}$ then 
\begin{align*}
X_{P'}X_j&=X_PX_{i_{m+1}}X_j=X_P(\lambda_{i_{m+1}j}-X_jX_{i_{m+1}})=\lambda_{i_{m+1}j}X_P-X_PX_jX_{i_{m+1}}\\
\overset{Ind.\ hyp.}&{=}\lambda_{i_{m+1}j}X_P-\left(\sum_{i \in P \setminus \lbrace j \rbrace} (-1)^{S(P \setminus \lbrace i,j \rbrace,P \setminus \lbrace j \rbrace)}\lambda_{ij} X_{P \setminus \lbrace i \rbrace} +(-1)^{|P|-\delta_{j \in P}}X_jX_P \right)X_{i_{m+1}}\\
&=\lambda_{i_{m+1}j}X_{P'\setminus \lbrace i_{m+1} \rbrace}+\sum_{i \in P \setminus \lbrace j \rbrace} (-1)^{S(P \setminus \lbrace i,j \rbrace,P \setminus \lbrace j \rbrace)+1}\lambda_{ij} X_{P' \setminus \lbrace i \rbrace} +(-1)^{|P'|-\delta_{j \in P'}}X_jX_{P'}\\
\overset{Lem.~\ref{lemmacrucial}, \ \ref{cor1})}&{=}\sum_{i \in P' \setminus \lbrace j \rbrace} (-1)^{S(P' \setminus \lbrace i,j \rbrace,P' \setminus \lbrace j \rbrace)}\lambda_{ij} X_{P' \setminus \lbrace i \rbrace} +(-1)^{|P'|-\delta_{j \in P'}}X_jX_{P'}.
\end{align*}
To prove \eqref{betacoroll} write $X_PX_j=X_{P_1}X_jX_{P_2}X_j$ and apply \eqref{XRXj} on $X_{P_2}X_j$.
\end{proof}

\subsection{Properties of the derivations associated to the canonical coaction}
Remember that the canonical $E(n)$-coaction on a Clifford algebra $A=Cl(\alpha,\beta_i,\gamma_i, \lambda_{ij})$ is associated to a $n+1$-uple of maps $(\sigma, d_1, \ldots, d_n)$, where $\sigma$ is the main involution of $A$ and each $\sigma$-derivation $d_i$ is determined by $d_i(X_j)=-\delta_{ij}$ for every $j=0,1,\ldots, n$. 

\begin{lemma}\label{lemmader}
If $d_i$ satisfy $d_i(X_j)=-\delta_{ij}$ for every $j=0,1,\ldots, n$, then 
\[d_i(X_R)=\delta_{i \in R}(-1)^{S(\lbrace i \rbrace, R)+1}X_{R \setminus \lbrace i \rbrace}\]
for every $R \subseteq \lbrace 0, 1, \ldots, n \rbrace$.
\end{lemma}
\begin{proof} The proof is by induction on $|R|$. We obviously have $d_i(1)=0=\delta_{i \in \emptyset}$ and $d_i(X_j)=-\delta_{ij}$ for every $j =0,1, \ldots, n$.
Now suppose 
\[d_i(X_R)=\delta_{i \in R}(-1)^{S(\lbrace i \rbrace, R)+1}X_{R \setminus \lbrace i \rbrace }\]
for all subsets $R$ of $ \lbrace 1, \ldots, n \rbrace$ with $|P|=r \geq 1$. Consider $R' \subseteq \lbrace 1, \ldots, n \rbrace$ with $|R'|=r+1$. Then we can write $R'=\lbrace i_1 < i_2 < \ldots < i_r<k \rbrace=R \sqcup \lbrace k \rbrace$ for $R=\lbrace i_1 < i_2 < \ldots < i_r \rbrace$. We have
\begin{align*}
d_i(X_{R'})&= \quad d_i(X_R)X_k+\sigma(X_R)d_i(X_k)\overset{Ind. \ hyp.}{=} \delta_{i \in R}(-1)^{S(\lbrace i \rbrace, R)+1}X_{R \setminus \lbrace i \rbrace }X_k+(-1)^{|R|}X_R(-\delta_{ik}).
\end{align*}
If $k=i$, then $i \notin R$ and we find $d_i(X_{R'})=(-1)^{|R|+1}X_R=\delta_{i \in R'}(-1)^{S(\lbrace i \rbrace, R')+1}X_{R' \setminus \lbrace i \rbrace }$. On the other hand, when $k \neq i$ we find
\[d_i(X_{R'})= \delta_{i \in R}(-1)^{S(\lbrace i \rbrace, R)+1}X_{R \setminus \lbrace i \rbrace }X_k=\delta_{i \in R'}(-1)^{S(\lbrace i \rbrace, R')+1}X_{R' \setminus \lbrace i \rbrace }.\]
\end{proof}

\begin{lemma}\label{lemmader2}
If each $d_i$ satisfy $d_i(X_j)=-\delta_{ij}$ for every $j=0,1,\ldots, n$, then 
\[d_P(X_R)=\delta_{P \subseteq R}(-1)^{S(P, R)+\frac{|P|(|P|+1)}{2}}X_{R \setminus P} \]
for every $P \subseteq \lbrace 1, \ldots, n \rbrace$, $R \subseteq \lbrace 0, 1, \ldots, n \rbrace$.
\end{lemma}

\begin{proof}
Let $P=\lbrace i_1 < i_2 < \ldots < i_{|P|} \rbrace$. If $P \not\subseteq R$, then there is a $j$, $1 \leq j \leq |P|$, such that $i_j \notin R$. Since $d_P=d_{i_1}d_{i_2} \circ \cdots \circ d_{i_{|P|}}$ and  $d_pd_q=-d_qd_p$ for every $1 \leq p,q \leq n$ we can write $d_P=(-1)^{S(P \setminus \lbrace j \rbrace,P)}d_{P \setminus \lbrace i_j \rbrace}d_{i_j}$ as it happens in \eqref{reorder}. Then
\[d_P(X_R)=(-1)^{S(P \setminus \lbrace j \rbrace,P)}d_{P \setminus \lbrace i_j \rbrace}d_{i_j}(X_R)=0\]
thanks to Lemma~\ref{lemmader}.

Now let $P \subseteq R$. If $P=\emptyset$, our thesis is trivially verified, while when $|P|=1$, it holds true by Lemma~\ref{lemmader}. Next, suppose that we have
\[d_P(X_R)=(-1)^{S(P, R)+\frac{|P|(|P|+1)}{2}}X_{R \setminus P}\]
for every $R \subseteq \lbrace 1, \ldots, n \rbrace$ and every $P \subseteq R$ with cardinality $|P|=m \geq 1$.
Consider a subset $P'$ of $R$ with cardinality $m+1$. We can write $P'=\lbrace i_1 < i_2 < \ldots < i_m<t \rbrace=P \sqcup \lbrace t \rbrace$, for some $P=\lbrace i_1 < i_2 < \ldots < i_m \rbrace \subseteq R$. Then we have
\begin{align*}
d_{P'}(X_R)&=d_P(d_t(X_R))\overset{Lem.~\ref{lemmader}}{=} d_P \left((-1)^{S(\lbrace t \rbrace, R)+1}X_{R \setminus \lbrace t \rbrace } \right)= (-1)^{S(\lbrace t \rbrace, R)+1}d_P \left(X_{R \setminus \lbrace t \rbrace } \right)\\
\overset{Ind. \ hyp.}&{=} (-1)^{S(\lbrace t \rbrace, R)+1+S(P, R \setminus \lbrace t \rbrace)+\frac{|P|(|P|+1)}{2}}X_{R \setminus P'}\\
&=(-1)^{S(\lbrace t \rbrace, R)+1+S(P' \setminus \lbrace t \rbrace, R \setminus \lbrace t \rbrace)+\frac{|P|(|P|+1)}{2}}X_{R \setminus P'}\\
\overset{Lem.~\ref{lemmacrucial}, \ \ref{lemmaFPF})}&{=} (-1)^{S(\lbrace t \rbrace, R)+1+S(P', R)+S(\lbrace t \rbrace, P')-S(\lbrace t \rbrace, R)+\frac{|P|(|P|+1)}{2}}X_{R \setminus P'}\\
&=(-1)^{S(P', R)+|P|+1+\frac{|P|(|P|+1)}{2}}X_{R \setminus P'}\\
&= (-1)^{S(P', R)+\frac{(|P'|)(|P'|+1)}{2}}X_{R \setminus P'}
\end{align*}
and the thesis follow from induction.
\end{proof}

\begin{corollary}\label{kerder}
If each $d_i$ satisfy $d_i(X_j)=-\delta_{ij}$ for every $j=0,1,\ldots, n$, then
\[\ker d_P=\left\lbrace \sum_{P \not\subseteq R \subseteq \lbrace 0,1, \ldots, n \rbrace} \mu_R X_R \ | \ \mu_R \in \Bbbk \right\rbrace\]
for every $P \subseteq \lbrace 1, \ldots, n \rbrace$.
\end{corollary}

\begin{notation}
For every $P\subseteq \lbrace 1, \ldots, n \rbrace$ let us indicate with 
\[\int a \ dX_P\]
the set of all elements $b \in A$ such that $d_P(b)=a$. 
\end{notation}

\begin{proposition}\label{propint}
If each $d_i$ satisfy $d_i(X_j)=-\delta_{ij}$ for every $j=0,1,\ldots, n$, then 
\[\int X_R \ d X_P=\begin{cases}
(-1)^{\frac{|P|(|P|+1)}{2}+|R||P|}X_R X_P+\sum_{Q \not\supseteq P}\mu_Q X_Q\quad &\textrm{if } P \cap R= \emptyset\\
\emptyset \quad &\textrm{if } P \cap R\neq \emptyset
\end{cases}\]
for every $P \subseteq \lbrace 1, \ldots, n \rbrace$, $R \subseteq \lbrace 0, 1, \ldots, n \rbrace$.
If both $\int a \ dX_P$ and $\int b \ dX_P$ are non-empty, then
\begin{align*}
\int c_1a+c_2b \ d X_P&=\left\lbrace c_1\omega_a+c_2\omega_b \ | \ \omega_a \in \int a \ dX_P, \omega_b \in \int b \ dX_P \right\rbrace\\
&=c_1\int a \ dX_P+c_2\int b \ dX_P
\end{align*}
for every  $c_1,c_2 \in \Bbbk$, $c_1c_2 \neq 0$.
\end{proposition}
\begin{proof}
Let us consider $P \subseteq \lbrace 1, \ldots, n \rbrace$ such that $P \cap R \neq \emptyset$ and suppose there is an element 
\[a \in \int X_R \ dX_P.\]
We write $a =\sum_{Q \subseteq \lbrace 0,1, \dots, n \rbrace}\mu_{Q}X_Q$. Then, by Lemma~\ref{lemmader2}, we have
\begin{align*}
X_R&=d_P(a)=\sum_{Q \subseteq \lbrace 0,1, \dots, n \rbrace}\mu_{Q}d_P(X_Q)=\sum_{Q \supseteq P}\mu_{Q}(-1)^{S(P, Q)+\frac{|P|(|P|+1)}{2}}X_{Q \setminus P}.
\end{align*}
Hence there must be a $Q_0 \supseteq P$ such that $Q_0 \setminus P=R$. It follows that $P \cap R=P \cap (Q_0 \setminus P) =\emptyset$, a contradiction.
Now consider a $P \subseteq \lbrace 1, \ldots, n \rbrace$ such that $P \cap R = \emptyset$. Using Lemma~\ref{lemmader2} we immediately see that
\[d_P \left( \sum_{Q \not\supseteq P}\mu_{Q}X_Q\right)=0\]
and that
\begin{align*}
d_P \left((-1)^{\frac{|P|(|P|+1)}{2}+|P||R|}X_RX_P\right)&=(-1)^{\frac{|P|(|P|+1)}{2}+|P||R|}d_P \left(X_RX_P\right)\\
&=(-1)^{\frac{|P|(|P|+1)}{2}+|P||R|}\sigma^{|P|}(X_R)d_P \left(X_P\right)\\
&=(-1)^{\frac{|P|(|P|+1)}{2}}X_Rd_P \left(X_P\right)\overset{Lem.~\ref{lemmader2}}{=}(-1)^{|P|(|P|+1)}X_R=X_R.
\end{align*}
Finally let $f_0:=(-1)^{\frac{|P|(|P|+1)}{2}+|P||R|}X_RX_P$ and $f_1$ be another element of the set $\int X_R \ dX_P$. We write
\[f_1-f_0=\sum_{Q \subseteq \lbrace 0, 1, \ldots, n \rbrace}\mu_Q X_Q.\]
Then
\begin{align*}
0=d_P(f_1-f_0)=\sum_{Q \subseteq \lbrace 0, 1, \ldots, n \rbrace}\mu_Q d_P(X_Q)\overset{Lem.~\ref{lemmader2}}{=}\sum_{Q \supseteq P}\mu_Q(-1)^{S(P, Q)+\frac{|P|(|P|+1)}{2}}X_{Q \setminus P},
\end{align*}
i.e. $\mu_Q=0$ for every $Q \supseteq P$, by linear independence.
This proves that
\[f_1=f_0+\sum_{Q \not\supseteq P}\mu_Q X_Q.\]

Clearly 
\[\int c_1a+c_2b \ d X_P \supseteq \left\lbrace c_1\omega_a+c_2\omega_b \ | \ \omega_a \in \int a \ dX_P, \omega_b \in \int b \ dX_P \right\rbrace.\]
Conversely, let $\omega \in \int c_1a+c_2b \ d X_P$. Then $d_P(\omega)=c_1a+c_2b$. Since $\int a \ d X_P$ and $\int b \ d X_P$ are non-empty, there are elements $\omega_u \in \int u \ d X_P$ for $u=a,b$. Then
\[d_P(\omega)=c_1d_P(\omega_a)+c_2d_P(\omega_b)=d_P(c_1 \omega_a + c_2 \omega_b)\]
and therefore $\omega-c_1\omega_a-c_2\omega_b=:\omega_0 \in \ker d_P$. Suppose $c_1 \neq 0$. Then $\omega=c_1\left( \omega_a+\frac{\omega_0}{c_1}\right)+c_2 \omega_b$ and $\omega_a+\frac{\omega_0}{c_1} \in \int a \ d X_P$. The results is similarly proved when $c_2 \neq 0$.
\end{proof}

\section{rt-separability of the cowreath \texorpdfstring{$(A \otimes E(n)^{op}, E(n), \psi)$}{} }\label{rtseparability}

Remember that, given an Hopf algebra $H$ and an $H$-comodule algebra $A$, we can build a cowreath $(A \otimes H^{op}, H, \psi)$ following the directions of Proposition~\ref{propXA} and that this cowreath is rt-separable if, and only if, the conditions of Proposition~\ref{PropReferee} are verified. We first show that the number of these conditions can be reduced.
\begin{lemma}\label{useful1}
A cowreath $(A\otimes H^{op},H,\psi )$ is rt-separable, i.e. separable via a
Casimir element $B:H\otimes H\rightarrow A\otimes H^{op}$
of the form
\begin{equation*}
h\otimes h^{\prime } \mapsto B^{A}\left( h\otimes h^{\prime }\right)
\otimes 1_{H}
\end{equation*}%
if and only if $B^{A}$ satisfies the following conditions
\begin{equation}
a_{0}B^{A}(1\otimes S(a_1)ha_{2})=B^{A}(1\otimes h)a
\label{B1S1}
\end{equation}%
\begin{equation}
B^{A}(1\otimes h)_{0}\otimes B^{A}(1\otimes h)_{1}=B^{A}(1\otimes h_{1})\otimes h_{2}  \label{B2S1}
\end{equation}%
\begin{equation}
B^A(h \otimes h')=B^A(1 \otimes S(h)h')\text{ and }B^{A}\left(
1_{H}\otimes 1_{H}\right) =1_{A}  \label{B3S1}
\end{equation}%
Moreover, whenever $S$ is invertible, it is rth-separable if and only if $B^A$
satisfies the further condition%
\begin{equation}
B^{A}\left( 1\otimes h\right) \cdot B^{A}\left( 1 \otimes h'
\right) =B^{A}\left( 1\otimes hh'\right)  \label{B4S1}
\end{equation}
\end{lemma}

\begin{proof}
Suppose $(A\otimes H^{op},H,\psi )$ is separable via the
Casimir element $B$ defined in the statement. By Proposition~\ref{PropReferee}, equalities \eqref{B1S}-\eqref{B3S} hold.
Then
\[B^A(h \otimes h')=B^A(h_1 \otimes h_2S(h_3)h')\overset{\eqref{B3S}}{=}B^A(1 \otimes S(h)h')\]
for any $h, h' \in H$, i.e. we have \eqref{B3S1}. Furthermore \eqref{B2S} coincides with \eqref{B2S1} and
\[a_0B^A(1 \otimes S(a_1)ha_2)\overset{\eqref{B3S1}}{=}a_0B^A(a_1 \otimes ha_2)\overset{\eqref{B1S}}{=}B^A(1 \otimes h)a,\]
so that also \eqref{B1S1} holds.
On the other hand, suppose \eqref{B1S1}-\eqref{B3S1} are verified. Then
\begin{align*}
B^A(h_1h' \otimes h_2h'')\overset{\eqref{B3S1}}&{=}B^A(1 \otimes S(h_1h')h_2h''=B^A(1 \otimes S(h')S(h_1)h_2h'')=\varepsilon(h)B^A(1 \otimes S(h')h'')\\
\overset{\eqref{B3S1}}&{=}\varepsilon(h)B^A(h' \otimes h'')
\end{align*}
for any $h,h',h'' \in H$. Therefore \eqref{B3S} holds. Once again we have that \eqref{B2S1} and \eqref{B2S} coincide and 
\begin{align*}
a_{0}B^{A}(ha_{1}\otimes h^{\prime }a_{2})\overset{\eqref{B3S1}}&{=}a_{0}B^{A}(1\otimes S(ha_{1})h^{\prime }a_{2})=a_{0}B^{A}(1\otimes S(a_1)S(h)h'a_{2})
\overset{\eqref{B1S1}}{=} B^{A}(1\otimes S(h)h')a\\
\overset{\eqref{B3S1}}&{=} B^{A}(h\otimes h')a,
\end{align*}
so that we get \eqref{B1S}. By Proposition~\ref{PropReferee} we can conclude that $(A\otimes H^{op},H,\psi )$ is rt-separable.

Now suppose the antipode $S$ is an invertible map. If $(A\otimes H^{op},H,\psi )$ is rth-separable, the first part of the proof and Proposition~\ref{PropReferee} tell us that \eqref{B1S1}-\eqref{B3S1} and \eqref{B4S} hold.
Then we have
\begin{align*}
B^A(1 \otimes h)\cdot B^A(1 \otimes h')\overset{\eqref{B3S1}}&{=}B^A(S^{-1}(h) \otimes 1)\cdot B^A(1 \otimes h')\overset{\eqref{B4S}}{=}B^A(S^{-1}(h) \otimes h')\overset{\eqref{B3S1}}{=}B^A(1 \otimes h h'),
\end{align*}
that is \eqref{B4S1}. Conversely, if \eqref{B1S1}-\eqref{B4S1} holds, then 
\begin{align*}
B^A(h \otimes 1) \cdot B^A(1 \otimes h')\overset{\eqref{B3S1}}&{=}B^A(1 \otimes S(h))\cdot B^A(1 \otimes h')\overset{\eqref{B4S1}}{=}B^A(1 \otimes S(h)h')\overset{\eqref{B3S1}}{=}B^A(h \otimes h'),
\end{align*}
i.e. \eqref{B1S}-\eqref{B4S} are satisfied and the cowreath is rth-separable.
\end{proof}

Let $\Bbbk$ be a fixed field of characteristic $\chara \Bbbk \neq 2$, $H=E(n)$ and $A$ be a finite-dimensional $H$-comodule algebra. Remember that the Hopf algebra $E(n)$ has $\Bbbk$-basis $\lbrace g^jx_P \ | \ P\subseteq \lbrace 1, \ldots ,n \rbrace, j \in \lbrace 0,1 \rbrace \rbrace$ and that to $\rho$ is associated a tuple of maps $(\varphi, d_1, \ldots, d_n)$ (see Theorem~\ref{maingen}). Let us also define $t_{k,Q}:=B^A(1 \otimes g^kx_Q)$ for $k=0,1$ and $Q \subseteq \lbrace 1, \ldots, n \rbrace$.
\begin{theorem}\label{theo1}
The cowreath $(A\otimes E(n)^{op},E(n),\psi )$ is rt-separable if, and only if,
\begin{equation}\label{B0E}
B^{A}\left(1\otimes 1\right) =1,
\end{equation}
\begin{equation}
B^A(g^jx_P \otimes g^kx_Q)=B^A(1 \otimes S(g^jx_P)g^kx_Q),  \label{B1E}
\end{equation}%
\begin{equation}\label{B2E}
\varphi(t_{k,Q})=(-1)^{k+|Q|}t_{k,Q},
\end{equation}
\begin{equation}\label{B3E}
d_P(t_{k,Q})=\begin{cases}
(-1)^{S(P,Q)+\frac{|P|(|P|+2k+1)}{2}} t_{k,Q \setminus P} &\textrm{if } P \subseteq Q, \\
0  &\textrm{if } P \not\subseteq Q,
\end{cases}  
\end{equation}
\begin{equation} \label{B4E}
t_{k,Q}a =\varphi^{|Q|}(a)t_{k,Q}+\delta_{k1}\underset{\emptyset \neq P \subseteq Q^c}{\sum}(-1)^{\frac{|P|(|P|+3)}{2}+S(P,P\sqcup Q)}2^{|P|}d_P\left(\varphi^{|P|+|Q|}(a)\right)t_{k,P \sqcup Q}
\end{equation}
for $j,k=0,1$ and every $P,Q \subseteq \lbrace 1, \ldots, n\rbrace$.
Moreover it is rth-separable if, and only if, in addition,
\begin{equation}  \label{B6S}
\begin{cases}
(t_{1,\emptyset})^2=1\\
t_{1,P}=(-1)^{|P|}t_{0,P}t_{1,\emptyset}
\end{cases}
\end{equation}
for every $P \subseteq \lbrace 1, \ldots, n\rbrace$.
\end{theorem}

\begin{proof}
Since $B^A$ is a linear map, it is completely determined by the values of $B^A(g^jx_P \otimes g^k x_Q)$, where as usual $g^jx_P$, $g^kx_Q$ denote the elements of the basis of $E(n)$. This means that in our case, \eqref{B3S1} is equivalent to \eqref{B0E} and \eqref{B1E}.

By linearity of $B^A$, \eqref{B2S1} holds true for every $h \in H$ if, and only if, it is verified for every element $g^kx_Q$ of the basis. This means that it is equivalent to 
\begin{equation}\label{rhot}
\rho(t_{k,Q})=(B^A \otimes \textrm{Id}_H)(1 \otimes \Delta(g^kx_Q)).
\end{equation}
for $k=0,1$, and every $Q \subseteq \lbrace 1, \ldots, n \rbrace$.
By means of \eqref{deltagx}, the RHS becomes equal to
\[(B^A \otimes \textrm{Id}_H) \left( 1 \otimes \sum_{F \subseteq Q} (-1)^{S(F,Q)} g^k x_F \otimes g^{|F|+k} x_{Q \setminus F} \right)= \sum_{F \subseteq Q} (-1)^{S(F,Q)} t_{k,F} \otimes g^{|F|+k} x_{Q \setminus F},\]
while, by employing \eqref{coactionfinal}, the LHS can be rewritten as
\[\underset{P}{\sum}(-1)^{\frac{|P|(|P|+1)}{2}} \left[ d_P\left(\frac{t_{k,Q}+(-1)^{|P|}\varphi(t_{k,Q})}{2}\right) \otimes x_P +d_P\left(\frac{(-1)^{|P|}t_{k,Q}-\varphi(t_{k,Q})}{2}\right) \otimes gx_P \right].\]
Then \eqref{rhot} becomes
\[
\begin{split}
&\underset{P}{\sum}(-1)^{\frac{|P|(|P|+1)}{2}} \left[ d_P\left(\frac{t_{k,Q}+(-1)^{|P|}\varphi(t_{k,Q})}{2}\right) \otimes x_P +d_P\left(\frac{(-1)^{|P|}t_{k,Q}-\varphi(t_{k,Q})}{2}\right) \otimes gx_P \right]=\\
&=\sum_{F \subseteq Q} (-1)^{S(F,Q)} t_{k,F} \otimes g^{|F|+k} x_{Q \setminus F}.
\end{split}
\]
By comparing elements with the same right tensorand we find
\begin{equation}\label{eq36}d_P\left(\frac{t_{k,Q}+(-1)^{|P|}\varphi(t_{k,Q})}{2}\right)=0=d_P\left(\frac{(-1)^{|P|}t_{k,Q}-\varphi(t_{k,Q})}{2}\right)
\end{equation}
for every $P \not \subseteq Q$
and
\begin{equation}\label{eq37}
\sum_{j=0,1}(-1)^{\frac{|P|(|P|+2j+1)}{2}}d_P\left(\frac{t_{k,Q}+(-1)^{|P|+j}\varphi(t_{k,Q})}{2}\right) \otimes g^jx_P   =(-1)^{S(Q \setminus P,Q)} t_{k,Q \setminus P} \otimes g^{|Q \setminus P|+k} x_P,
\end{equation}
for every $P \subseteq Q$. Equality \eqref{eq36} is equivalent to $d_P(t_{k,Q})=0$ for every $P \not \subseteq Q$, which is the second part of \eqref{B3E}.
By putting $P=\emptyset$  into \eqref{eq37} we obtain 
\begin{equation*}
\sum_{j=0,1}\frac{t_{k,Q}+(-1)^j\varphi(t_{k,Q})}{2} \otimes g^j = t_{k,Q} \otimes g^{|Q|+k},
\end{equation*}
i.e. \eqref{B2E}.
If we substitute it into \eqref{eq37} we get
\[(-1)^{\frac{|P|(|P|+1)}{2} +|P|(|P|+|Q|+k)}d_P\left(t_{k,Q}\right) \otimes g^{|P|+|Q|+k}x_P   =(-1)^{S(Q \setminus P,Q)} t_{k,Q \setminus P} \otimes g^{|Q \setminus P|+k} x_P\]
for every $P \subseteq Q$, i.e.
\begin{align*}
d_P\left(t_{k,Q}\right) &=(-1)^{S(Q \setminus P,Q)+|P|(|P|+|Q|+k)+\frac{|P|(|P|+1)}{2}} t_{k,Q \setminus P}\\
\overset{Lem.~\ref{lemmacrucial},\ \ref{corcompl})}&{=}(-1)^{S(P,Q)+k|P|+\frac{|P|(|P|+1)}{2}} t_{k,Q \setminus P}
\end{align*}
for every $P \subseteq Q$. Thus we can conclude that \eqref{eq37} is equivalent to \eqref{B2E} and
\[d_P\left(t_{k,Q}\right) =(-1)^{S(P,Q)+\frac{|P|(|P|+2k+1)}{2}} t_{k,Q \setminus P}\]
for every $P \subseteq Q$, that is the first part of \eqref{B3E}.

Next, we claim that
\begin{equation}\label{useful2}
B^A(1 \otimes S((x_P)_1)g^kx_Q(x_P)_2)=\delta_{k1}\delta_{P \subseteq Q^c}(-1)^{|P||Q|+S(P,P\sqcup Q)}t_{k,P \sqcup Q}2^{|P|}
\end{equation}
when $P \neq \emptyset$.
To prove this, observe that 
\begin{align*}
B^A(1 \otimes S((x_P)_1)g^kx_Q(x_P)_2)\overset{\eqref{deltagx}}&{=}\sum_{F \subseteq P} (-1)^{S(F,P)}B^A(1 \otimes S(x_F)g^kx_Qg^{|F|}x_{P\setminus F})\\
\overset{\eqref{antipode}}&{=}\sum_{F \subseteq P} (-1)^{S(F,P)}B^A(1 \otimes g^{|F|}x_Fg^kx_Qg^{|F|}x_{P\setminus F})\\
&=\sum_{F \subseteq P} (-1)^{S(F,P)+|P||Q|+|F|(|F|+k)}B^A(1 \otimes g^kx_Fx_{P\setminus F}x_Q)\\
\overset{\eqref{reorder}}&{=}\sum_{F \subseteq P} (-1)^{|P||Q|+|F|(|F|+k)}B^A(1 \otimes g^kx_Px_Q)\\
&=(-1)^{|P||Q|}B^A(1 \otimes g^kx_Px_Q)\sum_{F \subseteq P} (-1)^{|F|(|F|+k)}\\
&=(-1)^{|P||Q|}B^A(1 \otimes g^kx_Px_Q)\sum_{j=0}^{|P|} (-1)^{(k+1)j} \binom{|P|}{j}.\\
\end{align*}
If we assume $P \neq \emptyset$, the last term is equal to $\delta_{k1}(-1)^{|P||Q|}B^A(1 \otimes g^kx_Px_Q)2^{|P|}$ and thus we have proved our claim. We are going to use it to rewrite \eqref{B1S1} for $h=g^kx_Q$:
\begin{align*}
t_{k,Q}a&=a_0B^A(1 \otimes S(a_1)g^kx_Qa_2)\\
\overset{\eqref{coactionfinal}}&{=}\underset{P}{\sum}(-1)^{\frac{|P|(|P|+1)}{2}}  d_P\left(\frac{a+(-1)^{|P|}\varphi(a)}{2}\right) B^A(1 \otimes S((x_P)_1) g^kx_Q (x_P)_2)+\\
&+\underset{P}{\sum}(-1)^{\frac{|P|(|P|+1)}{2}}  d_P\left(\frac{(-1)^{|P|}a-\varphi(a)}{2}\right)B^A(1 \otimes S((x_P)_1)g^{k+1}x_Q g(x_P)_2)\\
&=\underset{P}{\sum}(-1)^{\frac{|P|(|P|+1)}{2}}  d_P\left(\frac{(1+(-1)^{|P|+|Q|})a+((-1)^{|P|}+(-1)^{|Q|+1})\varphi(a)}{2}\right) \cdot\\
&\cdot  B^A(1 \otimes S((x_P)_1) g^kx_Q (x_P)_2)\\
&=\underset{P}{\sum}(-1)^{\frac{|P|(|P|+1)}{2} +|P|(|Q|+1)}d_P\left(\varphi^{|P|+|Q|}(a)\right)B^A(1 \otimes S((x_P)_1) g^kx_Q (x_P)_2)\\
\overset{\eqref{useful2}}&{=}\varphi^{|Q|}(a)t_{k,Q}+\delta_{k1}\delta_{P \subseteq Q^c}\underset{P \neq \emptyset}{\sum}(-1)^{\frac{|P|(|P|+3)}{2}+S(P,P\sqcup Q)}2^{|P|}d_P\left(\varphi^{|P|+|Q|}(a)\right)t_{k,P \sqcup Q}\\
&=\varphi^{|Q|}(a)t_{k,Q}+\delta_{k1}\underset{\emptyset \neq P \subseteq Q^c}{\sum}(-1)^{\frac{|P|(|P|+3)}{2}+S(P,P\sqcup Q)}2^{|P|}d_P\left(\varphi^{|P|+|Q|}(a)\right)t_{k,P \sqcup Q},\\
\end{align*}
thus we have \eqref{B4E}. Clearly, if we assume \eqref{B4E} to hold, we find
\[a_0B^A(1 \otimes S(a_1)g^kx_Qa_2)=t_{k,Q}a=B^A(1 \otimes g^kx_Q)a,\]
just by moving to the end the first term in the sequence of equalities above. By $\Bbbk$-linearity of $B^A$ we conclude that \eqref{B4E} and \eqref{B1S1} are equivalent.

Finally, since $S$ is invertible, the cowreath $(A\otimes H^{op},H,\psi )$ is h-separable via the element $B$ if, and only if, it is separable and \eqref{B4S1} holds. Once again, by linearity of $B^A$, the last condition is equivalent to
\[B^A(1 \otimes g^jx_P)B^A(1 \otimes g^kx_Q)=B^A(1 \otimes g^jx_Pg^kx_Q)
\]
for $j,k=0,1$ and every $P,Q \subseteq \lbrace 1, \ldots, n \rbrace$. It is not hard to see that this equality is verified for every $Q$ if, and only if, it is verified for every $Q$ with cardinality $|Q| \leq 1$. 
When $|Q|=0$ and $k=0$ we get 
\[B^A(1 \otimes g^jx_P)B^A(1 \otimes 1)=B^A(1 \otimes g^jx_P),
\]
which holds in view of \eqref{B0E}. Then  \eqref{B4S1} is equivalent to
\[
\begin{cases}
t_{0,P}t_{1,\emptyset}&=(-1)^{|P|}t_{1,P}\\
t_{1,P}t_{1,\emptyset}&=(-1)^{|P|}t_{0,P}\\
t_{0,P}t_{0,\lbrace i \rbrace}&=\delta_{i \notin P}(-1)^{S(P,P \sqcup \lbrace i \rbrace)}t_{0,P \sqcup \lbrace i \rbrace}\\
t_{1,P}t_{0,\lbrace i \rbrace}&=\delta_{i \notin P}(-1)^{S(P,P \sqcup \lbrace i \rbrace)}t_{1,P \sqcup \lbrace i \rbrace}\\
t_{0,P}t_{1,\lbrace i \rbrace}&=\delta_{i \notin P}(-1)^{S(P,P \sqcup \lbrace i \rbrace)+|P|}t_{1,P \sqcup \lbrace i \rbrace}\\
t_{1,P}t_{1,\lbrace i \rbrace}&=\delta_{i \notin P}(-1)^{S(P,P \sqcup \lbrace i \rbrace)+|P|}t_{0,P \sqcup \lbrace i \rbrace}\\
\end{cases}
\]
for every $P,\lbrace i \rbrace \subseteq \lbrace 1, \ldots, n \rbrace$. When $P=\emptyset$ the second equality reads $(t_{1,\emptyset})^2=t_{0,\emptyset}=1$. Then, if we combine the first and the second equality we get
\[t_{0,P}=(-1)^{|P|}t_{1,P}t_{1,\emptyset}=t_{0,P}(t_{1,\emptyset})^2=t_{0,P},\]
which tell us that \eqref{B4S1} is also equivalent to
\[
\begin{cases}
t_{0,P}t_{1,\emptyset}&=(-1)^{|P|}t_{1,P}\\
(t_{1,\emptyset})^2&=1\\
t_{0,P}t_{0,\lbrace i \rbrace}&=\delta_{i \notin P}(-1)^{S(P,P \sqcup \lbrace i \rbrace)}t_{0,P \sqcup \lbrace i \rbrace}\\
t_{1,P}t_{0,\lbrace i \rbrace}&=\delta_{i \notin P}(-1)^{S(P,P \sqcup \lbrace i \rbrace)}t_{1,P \sqcup \lbrace i \rbrace}\\
t_{0,P}t_{1,\lbrace i \rbrace}&=\delta_{i \notin P}(-1)^{S(P,P \sqcup \lbrace i \rbrace)+|P|}t_{1,P \sqcup \lbrace i \rbrace}\\
t_{1,P}t_{1,\lbrace i \rbrace}&=\delta_{i \notin P}(-1)^{S(P,P \sqcup \lbrace i \rbrace)+|P|}t_{0,P \sqcup \lbrace i \rbrace}.\\
\end{cases}
\]
for every $P,\lbrace i \rbrace \subseteq \lbrace 1, \ldots, n \rbrace$.
Now we can use the first two conditions to rewrite the remaining four, so that everything boils down to
\begin{equation}\label{sys:5}
\begin{cases}
t_{1,P}&=(-1)^{|P|}t_{0,P}t_{1,\emptyset}\\
(t_{1,\emptyset})^2&=1\\
t_{0,P}t_{0,\lbrace i \rbrace}&=\delta_{i \notin P}(-1)^{S(P,P \sqcup \lbrace i \rbrace)}t_{0,P \sqcup \lbrace i \rbrace}\\
t_{1,\emptyset}t_{0,\lbrace i \rbrace}&=-t_{0,\lbrace i \rbrace}t_{1,\emptyset}.
\end{cases}
\end{equation}
for every $P,\lbrace i \rbrace \subseteq \lbrace 1, \ldots, n \rbrace$.
To conclude, we write \eqref{B4E} for $k=1$ and $a=t_{0,\lbrace i \rbrace}$. By means of \eqref{B2E} and \eqref{B3E} we find
\[t_{1,P}t_{0,\lbrace i \rbrace} =(-1)^{|P|}t_{0,\lbrace i \rbrace}t_{1,P}+\delta_{i \notin P}(-1)^{S(P,P \sqcup\lbrace i \rbrace)}2 t_{1,P \sqcup \lbrace i \rbrace}\]
and, thanks to the first equality in \eqref{sys:5},
\[t_{0,P}t_{1,\emptyset}t_{0,\lbrace i \rbrace} =t_{0,\lbrace i \rbrace}(-1)^{|P|}t_{0,P}t_{1,\emptyset}+\delta_{i \notin P}(-1)^{S(P,P \sqcup\lbrace i \rbrace)+1}2 t_{0,P \sqcup \lbrace i \rbrace}t_{1,\emptyset}.\]
If we put $P=\emptyset$ we find $t_{1,\emptyset}t_{0,\lbrace i \rbrace} =t_{0,\lbrace i \rbrace}t_{1,\emptyset}-2 t_{0,\lbrace i \rbrace}t_{1,\emptyset}$, which is the fourth equality in \eqref{sys:5}. On the other hand, if we multiply everything by $t_{1,\emptyset}$ and we use the second equality in \eqref{sys:5} we find
\[t_{0,P}t_{0,\lbrace i \rbrace} =t_{0,\lbrace i \rbrace}(-1)^{|P|+1}t_{0,P}+\delta_{i \notin P}(-1)^{S(P,P \sqcup\lbrace i \rbrace)}2 t_{0,P \sqcup \lbrace i \rbrace}.\]
Equality \eqref{B4E} for $k=0$ and $a=t_{0,\lbrace i \rbrace}$ gives $t_{0,P}t_{0,\lbrace i \rbrace}=(-1)^{|P|}t_{0,\lbrace i \rbrace}t_{0,P}$, thus the former condition rewrites as
 \[2t_{0,P}t_{0,\lbrace i \rbrace} =\delta_{i \notin P}(-1)^{S(P,P \sqcup\lbrace i \rbrace)}2 t_{0,P \sqcup \lbrace i \rbrace},\]
which is the third equality in \eqref{sys:5}.
\end{proof}

\section{The case of Clifford algebras}\label{Cliffordcase}

Let $A=Cl(\alpha,\beta_i, \gamma_i, \lambda_{ij})$ be a Clifford algebra endowed with the canonical $E(n)$-comodule algebra structure given by \eqref{canonical1}-\eqref{canonicalGXi}. We recall that the corresponding tuple of maps is given by $(\sigma, d_1, \ldots, d_n)$, where $\sigma$ is the main involution of $A$ and each $d_j$ is determined by $d_j(X_i)=-\delta_{ij}$ for every $i=0,1,\ldots, n$ (see Proposition~\ref{canonictuple}). 
In this case we can slightly rephrase \autoref{theo1}.
\begin{theorem}\label{theo2}
The cowreath $(A\otimes E(n)^{op},E(n),\psi )$ is rt-separable
if, and only if,
\begin{equation}\label{B0cl}
B^{A}\left(1 \otimes 1 \right) =1,
\end{equation}
\begin{equation}
B^A(g^jx_P \otimes g^kx_Q)=B^A(1 \otimes S(g^jx_P)g^kx_Q),  \label{B1cl}
\end{equation}%
\begin{equation}\label{B2cl}
\sigma(t_{k,Q})=(-1)^{k+|Q|}t_{k,Q},
\end{equation}
\begin{equation}\label{B3cl}
d_P(t_{k,Q})=\begin{cases}
(-1)^{S(P,Q)+\frac{|P|(|P|+2k+1)}{2}} t_{k,Q \setminus P} \quad &\textrm{if } P \subseteq Q, \\
0  \quad &\textrm{if } P \not\subseteq Q,
\end{cases}  
\end{equation}
\begin{equation} \label{B4cl}
t_{k,Q}G =(-1)^{|Q|}Gt_{k,Q},
\end{equation}
\begin{equation} \label{B5cl}
t_{k,Q}X_i =(-1)^{|Q|}X_it_{k,Q}+\delta_{k1}\delta_{i \notin Q}2(-1)^{S(Q,Q \sqcup \lbrace i \rbrace)} t_{k,Q \sqcup \lbrace i \rbrace}
\end{equation}
for $j,k=0,1$ and every $P,Q, \lbrace i \rbrace \subseteq \lbrace 1, \ldots, n\rbrace$.
Moreover it is rth-separable if, and only if, in addition,
\begin{equation}  \label{B6cl}
\begin{cases}
(t_{1,\emptyset})^2=1\\
t_{1,P}=(-1)^{|P|}t_{0,P}t_{1,\emptyset}.
\end{cases}
\end{equation}
for every $P \subseteq \lbrace 1, \ldots, n\rbrace$.
\end{theorem}
\begin{proof}
First of all notice that \eqref{B4cl} and \eqref{B5cl} are simply \eqref{B4E} on the generators $G$ and $X_i$ of the algebra $A$. Then the statement simply follows once we have shown that \eqref{B4E} is multiplicative with respect to $a$. To this aim it is enough to remember that \eqref{B4E} is equivalent to \eqref{B1S1}, which is clearly multiplicative with respect to $a$, since $A$ is an  $H$-comodule algebra.
\end{proof}

\subsection{The elements \texorpdfstring{$t_{0, Q}$}{}}\label{sec:t0Q}
Let us determine the exact form of the elements $t_{0,Q}$ for $|Q| \leq 1$. Clearly $t_{0,\emptyset}=B^A(1 \otimes 1)\overset{\eqref{B0cl}}{=}1$. Let $Q=\lbrace j \rbrace$. Then \eqref{B3cl} becomes 
\[
d_P(t_{0,\lbrace j \rbrace})=
\begin{cases}
-1 \quad &\textrm{for } P= \lbrace j \rbrace\\
0 \quad &\textrm{for all } P\subseteq \lbrace 1, \ldots n \rbrace, \ P \neq \emptyset, \lbrace j \rbrace
\end{cases}
\]
and in particular we have that $t_{0,\lbrace j \rbrace} \in \ker d_i$ for every $i=1, \ldots, n$, $i \neq j$. By Lemma~\ref{kerder}, this implies
\[t_{0,\lbrace j \rbrace}=\eta_{1j}1+\eta_{2j}G+\eta_{3j}X_j+\eta_{4j}GX_j,\]
while $d_j(t_{0,\lbrace j \rbrace})=-1$ forces $\eta_{3j}=1$ and $\eta_{4j}=0$. Also, equality \eqref{B2cl} yields $\eta_{1j}=0$, thus $t_{0,\lbrace j \rbrace}=\eta_{\lbrace j \rbrace}G+X_j$ for some $\eta_{\lbrace j \rbrace} \in \Bbbk$.
Now \eqref{B4cl} reads
\[\eta_{\lbrace j \rbrace}\alpha+X_jG=-\eta_{\lbrace j \rbrace}\alpha-GX_j,\]
that is $2\eta_{\lbrace j \rbrace} \alpha+\gamma_j=0$,
while \eqref{B5cl} becomes
\[\eta_{\lbrace j \rbrace}GX_i+X_jX_i=-\eta_{\lbrace j \rbrace}X_iG-X_iX_j,\]
i.e. $\eta_{\lbrace j \rbrace}\gamma_i+\lambda_{ij}=0$ if $i \neq j$ and $\eta_{\lbrace j \rbrace}\gamma_j+2\beta_j=0$ otherwise. In conclusion we must have $
t_{0,\lbrace j \rbrace}=\eta_{\lbrace j \rbrace}G+X_j$ with
\begin{equation}\label{t01}
\begin{cases}
\gamma_j+2 \alpha \eta_{\lbrace j \rbrace}=0\\
\lambda_{ij}+\eta_{\lbrace j \rbrace}\gamma_i=0\\
2\beta_j+\eta_{\lbrace j \rbrace}\gamma_j=0
\end{cases}
\quad \textrm{for every } i,j=1, \ldots, n, \ i \neq j.
\end{equation}
Notice that this agrees with the result proved by Menini and Torrecillas in \cite[Thm. 6.1]{mt} when $n=1$. In that case $B^A(1 \otimes x)=\eta G+X$ with $\eta \in \Bbbk$ satisfying $\gamma+2 \alpha \eta=0$ and $2\beta+\eta\gamma=0$.

\begin{proposition}\label{propt0Q}
Let $X_0=G$, $\overline{Q}:=Q \sqcup \lbrace 0 \rbrace$ and $t_{0,\emptyset}=1$. Then $t_{0,Q}$ satisfies \eqref{B2cl} and \eqref{B3cl} for every $Q \subseteq \lbrace 1, \ldots, n \rbrace$ if, and only if,
\begin{equation}\label{t0Q}
t_{0,Q}=\sum_{\underset{|R| \equiv_2 |Q|}{R \subseteq \overline{Q}}}(-1)^{S(Q \setminus R, Q)}\eta_{Q \setminus R} X_R
\end{equation}
where $\eta_{\emptyset}=1$ and $\eta_{S} \in \Bbbk$ for every $S \neq \emptyset$.
\end{proposition}

\begin{proof}
We start by proving that each $t_{0,Q}$ defined by \eqref{t0Q} satisfies \eqref{B2cl} and \eqref{B3cl}. First, we see that
\begin{align*}
\sigma(t_{0,Q})&=\sum_{\underset{|R| \equiv_2 |Q|}{R \subseteq \overline{Q}}}(-1)^{S(Q \setminus R, Q)}\eta_{Q \setminus R} \sigma(X_R)\\
&=\sum_{\underset{|R| \equiv_2 |Q|}{R \subseteq \overline{Q}}}(-1)^{S(Q \setminus R, Q)}\eta_{Q \setminus R} (-1)^{|R|}X_R\\
&=\sum_{\underset{|R| \equiv_2 |Q|}{R \subseteq \overline{Q}}}(-1)^{S(Q \setminus R, Q)}\eta_{Q \setminus R} (-1)^{|Q|}X_R\\
&=(-1)^{|Q|}\sum_{\underset{|R| \equiv_2 |Q|}{R \subseteq \overline{Q}}}(-1)^{S(Q \setminus R, Q)}\eta_{Q \setminus R} X_R\\
&=(-1)^{|Q|}t_{0,Q}
\end{align*}
so that \eqref{B2cl} holds. Then observe that for every $P \subseteq \lbrace 1, \ldots, n\rbrace$ we have 
\[d_P(t_{0,Q})=\sum_{\underset{|R| \equiv_2 |Q|}{R \subseteq \overline{Q}}}(-1)^{S(Q \setminus R, Q)}\eta_{Q \setminus R} d_P(X_R)\] which vanishes, by Lemma~\ref{lemmader2}, if $P \not\subseteq Q$. On the other hand, if $P \subseteq Q$ we have
\begin{align*}
d_P(t_{0,Q})&=\sum_{\underset{|R| \equiv_2 |Q|}{R \subseteq \overline{Q}}}(-1)^{S(Q \setminus R, Q)}\eta_{Q \setminus R} d_P(X_R)\\\overset{Lem.~\ref{lemmader2}}&{=}\sum_{\underset{|R| \equiv_2 |Q|}{P \subseteq R \subseteq \overline{Q}}}(-1)^{S(Q \setminus R, Q)}\eta_{Q \setminus R} (-1)^{S(P,R)+\frac{|P|(|P|+1)}{2}}X_{R\setminus P}\\
&=\sum_{\underset{|R'| \equiv_2 |Q\setminus P|}{R' \subseteq \overline{Q}\setminus P}}(-1)^{S(Q \setminus (R'\sqcup P), Q)+S(P,R' \sqcup P) +\frac{|P|(|P|+1)}{2}}\eta_{Q \setminus (R'\sqcup P)}X_{R'}\\
\overset{Lem.~\ref{lemmacrucial}, \ \ref{lemmaFPF})}&{=}\sum_{\underset{|R'| \equiv_2 |Q\setminus P|}{R' \subseteq \overline{Q}\setminus P}}(-1)^{S(Q \setminus (R'\sqcup P), Q \setminus R')+S(Q \setminus R',Q) +\frac{|P|(|P|+1)}{2}}\eta_{Q \setminus (R'\sqcup P)}X_{R'}\\
\overset{Lem.~\ref{lemmacrucial}, \ \ref{corcompl})}&{=}\sum_{\underset{|R'| \equiv_2 |Q\setminus P|}{R' \subseteq \overline{Q}\setminus P}}(-1)^{S(P, Q \setminus R')+|P|(|Q|-|R'|-|P|)+S(Q \setminus R',Q) +\frac{|P|(|P|+1)}{2}}\eta_{Q \setminus (R'\sqcup P)}X_{R'}\\
&=\sum_{\underset{|R'| \equiv_2 |Q\setminus P|}{R' \subseteq \overline{Q}\setminus P}}(-1)^{S(P, Q \setminus R')+S(Q \setminus R',Q) +\frac{|P|(|P|+1)}{2}}\eta_{Q \setminus (R'\sqcup P)}X_{R'}\\
\overset{Lem.~\ref{lemmacrucial}, \ \ref{lemmaFPF})}&{=}\sum_{\underset{|R'| \equiv_2 |Q\setminus P|}{R' \subseteq \overline{Q}\setminus P}}(-1)^{S(Q \setminus (R' \sqcup P), Q \setminus P)+S(P,Q) +\frac{|P|(|P|+1)}{2}}\eta_{Q \setminus (R'\sqcup P)}X_{R'}\\
&=(-1)^{S(P,Q) +\frac{|P|(|P|+1)}{2}}t_{0, Q \setminus P},
\end{align*}
so that \eqref{B3cl} is verified.

Now let us show that if every $t_{0,Q}$ satisfies \eqref{B2cl} and \eqref{B3cl}, then they must be of the form \eqref{t0Q}. The proof is by induction on $|Q|$. If $Q=\emptyset$ we have
\[t_{0,\emptyset}=1=(-1)^{S(\emptyset, \emptyset)}\eta_{\emptyset} X_{\emptyset}\]
and the statement holds true.
If $Q=\lbrace j \rbrace$, then
\begin{align*}
t_{0,Q}&=\sum_{\underset{|R| \ odd}{R \subseteq \lbrace 0,j \rbrace}}(-1)^{S(\lbrace j \rbrace \setminus R, \lbrace j \rbrace)}\eta_{\lbrace j \rbrace \setminus R} X_R\\
&=(-1)^{S(\lbrace j \rbrace, \lbrace j \rbrace)}\eta_{\lbrace j \rbrace} X_0+(-1)^{S(\emptyset, \lbrace j \rbrace)}\eta_{\emptyset} X_j\\
&=\eta_{\lbrace j \rbrace}G+X_j
\end{align*}
which again is true (see the beginning of this subsection).
Now suppose \eqref{t0Q} holds true for any $Q \subseteq \lbrace 1, \ldots, n \rbrace$ with cardinality $|Q|=m$ for a fixed $1 \leq m \leq n$. Consider $Q'=\lbrace i_1 < i_2 < \ldots <i_m < i_{m+1} \rbrace\subseteq \lbrace 1, \ldots, n \rbrace$ with cardinality $m+1$ and split it as $Q'=Q \sqcup \lbrace i_{m+1} \rbrace$, where $Q=\lbrace i_1 < i_2 < \ldots <i_m \rbrace$.
Then \eqref{B3cl} tells us that $d_{i_{m+1}}(t_{0,Q'})=(-1)^{|Q|+1}t_{0,Q}$, i.e. that $(-1)^{|Q|+1}t_{0,Q'} \in \int t_{0,Q} \ dX_{i_{m+1}}$. Hence we can use our induction hypothesis and, thanks to Proposition~\ref{propint}, we can write
\begin{align*}
(-1)^{|Q|+1}t_{0,Q'} \in \sum_{\underset{|R| \equiv_2 |Q|}{R \subseteq \overline{Q}}}(-1)^{S(Q \setminus R, Q)}\eta_{Q \setminus R} \int X_R \ d X_{i_{m+1}}
\end{align*}
and
\[(-1)^{|Q|+1}t_{0,Q'}= \sum_{\underset{|R| \equiv_2 |Q|}{R \subseteq \overline{Q}}}(-1)^{S(Q \setminus R, Q)}\eta_{Q \setminus R}\left[(-1)^{|R|+1}X_{R}X_{i_{m+1}}+\sum_{T \not\ni i_{m+1}} \mu_TX_T\right],\]
which (after a suitable renaming of the $\mu_T$'s) can be rewritten as
\[t_{0,Q'}= \sum_{\underset{|R| \equiv_2 |Q|}{R \subseteq \overline{Q}}}(-1)^{S(Q \setminus R, Q)}\eta_{Q \setminus R}X_{R}X_{i_{m+1}}+\sum_{T \not\ni i_{m+1}} \mu_TX_T.\]
Since by \eqref{B3cl} we know that $d_P(t_{0,Q'})=0$ for every $P \not\subseteq Q'$, we can immediately deduce that $\mu_T$=0 for every $T \not \subseteq \overline{Q}$, while \eqref{B2cl} forces $\mu_T=0$ for every $T$ with $|T| \equiv |Q|$. Hence we find
\begin{align*}
t_{0,Q'}&= \sum_{\underset{|R| \equiv_2 |Q|}{R \subseteq \overline{Q}}}(-1)^{S(Q \setminus R, Q)}\eta_{Q \setminus R}X_{R}X_{i_{m+1}}+\sum_{\underset{|T| \equiv_2 |Q|+1}{T \subseteq \overline{Q}}} \mu_TX_T\\
&= \sum_{\underset{|R'| \equiv_2 |Q'|}{i_{m+1} \in R' \subseteq \overline{Q'}}}(-1)^{S(Q' \setminus R', Q')}\eta_{Q' \setminus R'}X_{R'}+\sum_{\underset{|T| \equiv_2 |Q'|}{i_{m+1} \notin T \subseteq \overline{Q'}}} \mu_TX_T.
\end{align*}
To conclude, it is enough to prove that $\mu_T=(-1)^{S(Q'\setminus T, Q')}\eta_{Q'\setminus T}$ for every $T \subseteq \overline{Q}$ with $|T| \equiv_2 |Q|+1$. Fix a $T_0 \subseteq \overline{Q}$ with $|T_0| \equiv_2 |Q|+1$.
If we apply $d_{T_0}$ on the last equality we find
\begin{align*}
d_{T_0}(t_{0,Q'})&= \sum_{\underset{|R'| \equiv_2 |Q'|}{i_{m+1} \in R' \subseteq \overline{Q'}}}(-1)^{S(Q' \setminus R', Q')}\eta_{Q' \setminus R'}d_{T_0}(X_{R'})+\sum_{\underset{|T| \equiv_2 |Q'|}{i_{m+1} \notin T \subseteq \overline{Q'}}} \mu_Td_{T_0}(X_T)\\
\overset{Lem.~\ref{lemmader2}}&{=} (-1)^{\frac{|T_0|(|T_0|+1)}{2}}\mu_{T_0}+ \underset{\textrm{ terms with higher degree}}{\ldots}.
\end{align*}
On the other hand, by means of \eqref{B3cl}, we get
\begin{align*}
d_{T_0}(t_{0,Q'})&=(-1)^{S(T_0,Q')+\frac{|T_0|(|T_0|+1)}{2}}t_{0,Q'\setminus T_0}\\
\overset{Ind. \ hyp.}&{=}(-1)^{S(T_0,Q')+\frac{|T_0|(|T_0|+1)}{2}}\sum_{\underset{|R| \ even}{R \subseteq \overline{Q'}\setminus T_0}}(-1)^{S(Q' \setminus (R \sqcup T_0), Q'\setminus T_0)}\eta_{Q' \setminus (R \sqcup T_0)} X_R
\end{align*}
and the coefficient of the zero-degree term in this sum is $(-1)^{S(T_0,Q')+\frac{|T_0|(|T_0|+1)}{2}}\eta_{Q' \setminus T_0}$.
Therefore we must have
\begin{align*}
(-1)^{\frac{|T_0|(|T_0|+1)}{2}}\mu_{T_0}&=(-1)^{S(T_0,Q')+\frac{|T_0|(|T_0|+1)}{2}}\eta_{Q' \setminus T_0}\\
\overset{Lem.~\ref{lemmacrucial}, \ \ref{corcompl})}&{=}(-1)^{S(Q'\setminus T_0,Q')+|T_0|(|Q'|-|T_0|)+\frac{|T_0|(|T_0|+1)}{2}}\eta_{Q' \setminus T_0},
\end{align*}
that is $\mu_{T_0}=(-1)^{S(Q' \setminus T_0,Q')}\eta_{Q' \setminus T_0}$, thanks to the fact that $|T_0| \equiv_2 |Q'|$.
\end{proof}

\begin{proposition}\label{propt0Q2}
The elements $t_{0,Q}$ defined by \eqref{t0Q} satisfy \eqref{B4cl} and \eqref{B5cl} for every $Q \subseteq \lbrace 1, \ldots, n \rbrace$ and every $i=1, \ldots, n$ if, and only if, either
\begin{itemize}
\item $A=Cl(0,0,0,0)$ \\
or
\item $A=Cl\left(\alpha,\frac{\gamma_i^2}{4\alpha}, \gamma_i, \frac{\gamma_i\gamma_j}{2 \alpha}\right)$ with $\alpha \neq 0$ and $\eta_T =\sum_{i \in T}(-1)^{S(\lbrace i \rbrace, T)+1}\eta_{T \setminus  \lbrace i \rbrace}\frac{\gamma_i}{2\alpha}$ for every $T \subseteq Q$, with $|T|$ odd.
\end{itemize}
\end{proposition}

\begin{proof}
Condition \eqref{B4cl} can be rewritten as $t_{0,Q}X_0+(-1)^{|Q|+1}X_0t_{0,Q}=0$, therefore \eqref{B4cl} and \eqref{B5cl} together are equivalent to 
\begin{equation}\label{cond1}
t_{0,Q}X_i+(-1)^{|Q|+1}X_it_{0,Q}=0 \quad \textrm{ for every } i=0,1,\ldots, n.
\end{equation}
Observe that, for any of these values of $i$, we have
\begin{align*}
&t_{0,Q}X_i+(-1)^{|Q|+1}X_it_{0,Q}\\
&=\sum_{\underset{|R| \equiv_2 |Q|}{R \subseteq \overline{Q}}}(-1)^{S(Q \setminus R, Q)}\eta_{Q \setminus R} (X_RX_i+(-1)^{|Q|+1}X_iX_R)\\
&=\sum_{\underset{|R| \equiv_2 |Q|}{i \in R \subseteq \overline{Q}}}(-1)^{S(Q \setminus R, Q)}\eta_{Q \setminus R} (X_RX_i+(-1)^{|Q|+1}X_iX_R)+\\
&+\sum_{\underset{|R| \equiv_2 |Q|}{i \notin R \subseteq \overline{Q}}}(-1)^{S(Q \setminus R, Q)}\eta_{Q \setminus R} (X_RX_i+(-1)^{|Q|+1}X_iX_R)\\
\overset{\eqref{XRXj}}&{=}\sum_{\underset{|R| \equiv_2 |Q|}{i \in R \subseteq \overline{Q}}}(-1)^{S(Q \setminus R, Q)}\eta_{Q \setminus R} \left(2X_RX_i+\sum_{j \in R \setminus \lbrace i \rbrace} (-1)^{S(R \setminus \lbrace i,j \rbrace,R \setminus \lbrace i \rbrace)+1}\lambda_{ij} X_{R \setminus \lbrace j \rbrace}\right)+\\
&+\sum_{\underset{|R| \equiv_2 |Q|}{i \notin R \subseteq \overline{Q}}}(-1)^{S(Q \setminus R, Q)}\eta_{Q \setminus R} \sum_{j \in R} (-1)^{S(R \setminus \lbrace j \rbrace,R)}\lambda_{ij} X_{R \setminus \lbrace j \rbrace}\\
\overset{\eqref{betacoroll}}&{=}\sum_{\underset{|R| \equiv_2 |Q|}{i \in R \subseteq \overline{Q}}}2(-1)^{S(Q \setminus R, Q)}\eta_{Q \setminus R} \left(\sum_{\underset{j \in R}{j>i}} (-1)^{S(R \setminus \lbrace j \rbrace,R)}\lambda_{ij} X_{R \setminus \lbrace j \rbrace} +(-1)^{S(R \setminus \lbrace i \rbrace,R)}\beta_i X_{R \setminus \lbrace i \rbrace}\right)+\\
&+\sum_{\underset{|R| \equiv_2 |Q|}{i \in R \subseteq \overline{Q}}}(-1)^{S(Q \setminus R, Q)}\eta_{Q \setminus R} \left(\sum_{j \in R \setminus \lbrace i \rbrace} (-1)^{S(R \setminus \lbrace i,j \rbrace,R \setminus \lbrace i \rbrace)+1}\lambda_{ij} X_{R \setminus \lbrace j \rbrace}\right)+\\
&+\sum_{\underset{|R| \equiv_2 |Q|}{i \notin R \subseteq \overline{Q}}}(-1)^{S(Q \setminus R, Q)}\eta_{Q \setminus R} \sum_{j \in R} (-1)^{S(R \setminus \lbrace j \rbrace,R)}\lambda_{ij} X_{R \setminus \lbrace j \rbrace}.
\end{align*}
Since for any pair of distinct $i$ and $j$
\[S(R\setminus \lbrace i,j \rbrace,R\setminus \lbrace i \rbrace)\overset{Lem.\ref{lemmacrucial}, \ \ref{lemmaFPF})}{=}S(R\setminus \lbrace j \rbrace,R)-S(\lbrace i \rbrace,R)+S(\lbrace i \rbrace,R \setminus \lbrace j \rbrace)=S(R\setminus \lbrace j \rbrace,R)-\delta_{j<i},
\]
we obtain
\begin{align*}
&t_{0,Q}X_i+(-1)^{|Q|+1}X_it_{0,Q}\\
&=\sum_{\underset{|R| \equiv_2 |Q|}{i \in R \subseteq \overline{Q}}}(-1)^{S(Q \setminus R, Q)}\eta_{Q \setminus R} 2\left(\sum_{\underset{j \in R}{j>i}} (-1)^{S(R \setminus \lbrace j \rbrace,R)}\lambda_{ij} X_{R \setminus \lbrace j \rbrace} +(-1)^{S(R \setminus \lbrace i \rbrace,R)}\beta_i X_{R \setminus \lbrace i \rbrace}\right)+\\
&+\sum_{\underset{|R| \equiv_2 |Q|}{i \in R \subseteq \overline{Q}}}(-1)^{S(Q \setminus R, Q)}\eta_{Q \setminus R} \left(\sum_{j \in R \setminus \lbrace i \rbrace} (-1)^{S(R\setminus \lbrace j \rbrace,R)-\delta_{j<i}+1}\lambda_{ij} X_{R \setminus \lbrace j \rbrace}\right)+\\
&+\sum_{\underset{|R| \equiv_2 |Q|}{i \notin R \subseteq \overline{Q}}}(-1)^{S(Q \setminus R, Q)}\eta_{Q \setminus R} \sum_{j \in R} (-1)^{S(R \setminus \lbrace j \rbrace,R)}\lambda_{ij} X_{R \setminus \lbrace j \rbrace}\\
&=\sum_{\underset{|R| \equiv_2 |Q|}{i \in R \subseteq \overline{Q}}}(-1)^{S(Q \setminus R, Q)}\eta_{Q \setminus R} \left(\sum_{\underset{j \in R}{j>i}} (-1)^{S(R \setminus \lbrace j \rbrace,R)}\lambda_{ij} X_{R \setminus \lbrace j \rbrace} +2(-1)^{S(R \setminus \lbrace i \rbrace,R)}\beta_i X_{R \setminus \lbrace i \rbrace}\right)+\\
&+\sum_{\underset{|R| \equiv_2 |Q|}{i \in R \subseteq \overline{Q}}}(-1)^{S(Q \setminus R, Q)}\eta_{Q \setminus R} \left(\sum_{j<i} (-1)^{S(R\setminus \lbrace j \rbrace,R)}\lambda_{ij} X_{R \setminus \lbrace j \rbrace}\right)+\\
&+\sum_{\underset{|R| \equiv_2 |Q|}{i \notin R \subseteq \overline{Q}}}(-1)^{S(Q \setminus R, Q)}\eta_{Q \setminus R} \sum_{j \in R} (-1)^{S(R \setminus \lbrace j \rbrace,R)}\lambda_{ij} X_{R \setminus \lbrace j \rbrace}\\
&=\sum_{\underset{|R| \equiv_2 |Q|}{R \subseteq \overline{Q}}}(-1)^{S(Q \setminus R, Q)}\eta_{Q \setminus R} \sum_{j \in R} (-1)^{S(R \setminus \lbrace j \rbrace,R)}\lambda_{ij} X_{R \setminus \lbrace j \rbrace},\\
\end{align*}
where we introduce the notation $\lambda_{ii}:=2\beta_i$ for every $i=0,1, \ldots, n$.
Then \eqref{cond1} rewrites as
\[
\sum_{\underset{|R| \equiv_2 |Q|}{R \subseteq \overline{Q}}}(-1)^{S(Q \setminus R, Q)}\eta_{Q \setminus R} \sum_{j \in R} (-1)^{S(R \setminus \lbrace j \rbrace,R)}\lambda_{ij} X_{R \setminus \lbrace j \rbrace}=0,
\]
which can be rearranged as
\[
\sum_{\underset{|R'| \not\equiv_2 |Q|}{R' \subsetneq \overline{Q}}}\sum_{j \in \overline{Q}\setminus R' }(-1)^{S(Q \setminus (R' \sqcup \lbrace j \rbrace), Q)+S(R' ,R' \sqcup \lbrace j \rbrace)}\lambda_{ij} \eta_{Q \setminus (R' \sqcup \lbrace j \rbrace)}X_{R'}=0.
\]
Now it is clear that \eqref{cond1} is equivalent to
\[\sum_{j \in \overline{Q}\setminus R' }(-1)^{S(Q \setminus (R' \sqcup \lbrace j \rbrace), Q)+S(R' ,R' \sqcup \lbrace j \rbrace)}\lambda_{ij} \eta_{Q \setminus (R' \sqcup \lbrace j \rbrace)}=0 \]
for every $R' \subseteq \overline{Q}$ such that $|R'| \not\equiv_2 |Q|$ and every $i=0, 1, \ldots, n$. This can be split into two independent equalities according to whether $0$ is or is not contained in $R'$: in the first case we rewrite $R'=R \sqcup \lbrace 0 \rbrace$, while in the second we identify $R'$ with the corresponding subset $R$ of $Q$. We obtain
\begin{equation*}
\begin{cases}
\sum_{j \in Q \setminus R }(-1)^{S(Q \setminus (R \sqcup \lbrace j \rbrace), Q)+S(R, R \sqcup \lbrace j \rbrace)}\lambda_{ij} \eta_{Q \setminus (R \sqcup \lbrace j \rbrace)}=0 \quad &\textrm{for every } R  \subsetneq Q,\ |R| \equiv_2 |Q|\\
\sum_{j \in Q \setminus R}(-1)^{S(Q \setminus (R \sqcup \lbrace j \rbrace), Q)+S(R ,R \sqcup \lbrace j \rbrace)}\lambda_{ij} \eta_{Q \setminus (R \sqcup \lbrace j \rbrace)}+\\
+(-1)^{S(Q \setminus R, Q)+|R|}\gamma_i \eta_{Q \setminus R}=0 \quad &\textrm{for every } R \subseteq Q,\ |R| \not\equiv_2 |Q|
\end{cases}
\end{equation*}
for every $i=0, 1, \ldots, n$ and finally, by operating the substitution $T=Q\setminus R$, we find
\begin{equation*}
\begin{cases}
\sum_{j \in T }(-1)^{S(T \setminus \lbrace j \rbrace), Q)+S(Q \setminus T, (Q \setminus T) \sqcup \lbrace j \rbrace)}\lambda_{ij} \eta_{T \setminus  \lbrace j \rbrace}=0 \quad &\textrm{for every } \emptyset \neq T  \subseteq Q,\ |T| \textrm{ even}\\
\sum_{j \in T}(-1)^{S(T \setminus \lbrace j \rbrace), Q)+S(Q \setminus T, (Q \setminus T) \sqcup \lbrace j \rbrace)}\lambda_{ij} \eta_{T \setminus \lbrace j \rbrace}+\\
+(-1)^{S(T, Q)+|Q|-|T|}\gamma_i \eta_T=0 \quad &\textrm{for every } T \subseteq Q,\ |T| \textrm{ odd}
\end{cases}
\end{equation*}
for every $i=0, 1, \ldots, n$. Now observe that
\begin{align*}
&S(T \setminus \lbrace j \rbrace, Q)+S(Q \setminus T,(Q \setminus T) \sqcup \lbrace j \rbrace)=\\
\overset{Lem.~\ref{lemmacrucial}, \ \ref{corcompl})}&{=}S(T \setminus \lbrace j \rbrace, Q)+|Q|-|T|-S(\lbrace j \rbrace,(Q \setminus T) \sqcup \lbrace j \rbrace)\\
\overset{Lem.~\ref{lemmacrucial}, \ \ref{lemmaFPF})}&{=}2S(T \setminus \lbrace j \rbrace, Q)+|Q|-|T|-S(T,Q)-S(T \setminus \lbrace j \rbrace,T)\\
\overset{Lem.~\ref{lemmacrucial}, \ \ref{lemmaPi})}&{=}2S(T \setminus \lbrace j \rbrace, Q)+|Q|-2|T|-S(T,Q)+1+S(\lbrace j \rbrace, T).
\end{align*}
Hence we can conclude that \eqref{cond1} is equivalent to
\begin{equation}\label{sys:2}
\begin{cases}
\sum_{j \in T }(-1)^{S(\lbrace j \rbrace, T)}\lambda_{ij} \eta_{T \setminus  \lbrace j \rbrace}=0 \quad &\textrm{for every } \emptyset \neq T  \subseteq Q,\ |T| \textrm{ even}\\
\sum_{j \in T}(-1)^{S(\lbrace j \rbrace, T)}\lambda_{ij} \eta_{T \setminus \lbrace j \rbrace}+\gamma_i \eta_T=0 \quad &\textrm{for every } T \subseteq Q,\ |T| \textrm{ odd}
\end{cases}
\end{equation}
for every $i=0, 1, \ldots, n$. Notice that for $T=\lbrace j \rbrace$ we find $\lambda_{ij}+\eta_{\lbrace j \rbrace}\gamma_i =0$ for every $i=0,1,\ldots, n$ and every $j=1, \ldots n$, which is exactly \eqref{t01}. If $\alpha=0$, then \eqref{t01} immediately yields $\alpha=\beta_i=\gamma_i=\lambda_{ij}=0$ for every $i,j=1, \ldots, n$ and we find that $A=Cl(0,0,0,0)$, while \eqref{sys:2} becomes trivial. On the other hand, if $\alpha \neq 0$, then \eqref{t01} implies that $\eta_{\lbrace j \rbrace}=-\frac{\gamma_j}{2 \alpha}$, $\beta_j=\frac{\gamma_j^2}{4 \alpha}$ and $\lambda_{ij}=\frac{\gamma_i\gamma_j}{2\alpha}$ for every $i,j=1, \ldots, n$, $i<j$. This means that $A$ is of the type $A=\left(\alpha,\frac{\gamma_i^2}{4\alpha}, \gamma_i, \frac{\gamma_i\gamma_j}{2 \alpha}\right)$ and that \eqref{sys:2} changes into
\begin{equation}\label{sys:3}
\begin{cases}
\sum_{j \in T }(-1)^{S(\lbrace j \rbrace, T)}\frac{\gamma_i\gamma_j}{2 \alpha} \eta_{T \setminus  \lbrace j \rbrace}=0 \quad &\textrm{for every } \emptyset \neq T  \subseteq Q,\ |T| \textrm{ even}\\
\gamma_i\left(\sum_{j \in T}(-1)^{S(\lbrace j \rbrace, T)}\frac{\gamma_j}{2 \alpha} \eta_{T \setminus \lbrace j \rbrace}+ \eta_T\right)=0 \quad &\textrm{for every } T \subseteq Q,\ |T| \textrm{ odd}
\end{cases}
\end{equation}
for every $i=0, 1, \ldots, n$. In particular, for $i=0$, we have
\[2\alpha\left(\sum_{j \in T}(-1)^{S(\lbrace j \rbrace, T)}\frac{\gamma_j}{2 \alpha} \eta_{T \setminus \lbrace j \rbrace}+ \eta_T\right)=0 \quad \textrm{for every } T \subseteq Q,\ |T| \textrm{ odd},\]
i.e. $\eta_T=\sum_{j \in T}(-1)^{S(\lbrace j \rbrace, T)+1}\frac{\gamma_j}{2 \alpha} \eta_{T \setminus \lbrace j \rbrace}$ for every $T \subseteq Q$, with $|T|$ odd, since $\alpha \neq 0$. To complete our proof it is sufficient to show that this equality implies the first one of \eqref{sys:3} for every $i=0,1,\ldots, n$.
To this aim, observe that every $T \subseteq Q$ with odd cardinality can be written as $T' \setminus \lbrace j \rbrace$ for a non-empty $T' \subseteq Q$ with even cardinality and some $j \in T$. Then we have
\[
\eta_{T'\setminus \lbrace j \rbrace} =\sum_{l \in T'\setminus \lbrace j \rbrace}(-1)^{S(\lbrace l \rbrace, T'\setminus \lbrace j \rbrace)+1}\frac{\gamma_l}{2\alpha}\eta_{T' \setminus  \lbrace j,l \rbrace}
\]
and
\begin{align*}
\sum_{j \in T'}(-1)^{S(\lbrace j \rbrace, T')}\frac{\gamma_i\gamma_j}{2\alpha}\eta_{T' \setminus  \lbrace j \rbrace} =\sum_{j \in T'}\sum_{l \in T'\setminus \lbrace j \rbrace}(-1)^{S(\lbrace j \rbrace, T')+S(\lbrace l \rbrace, T'\setminus \lbrace j \rbrace)+1}\frac{\gamma_i\gamma_j\gamma_l}{4\alpha^2}\eta_{T' \setminus  \lbrace j,l \rbrace}
\end{align*}
for every $\emptyset \neq T' \subseteq Q$ with $|T'|$ even and every $i=0,1,\ldots, n$.
We point out that
\begin{align*}
S(\lbrace j \rbrace, T')+S(\lbrace l \rbrace, T' \setminus \lbrace j \rbrace) \overset{Lem.~\ref{lemmacrucial} \ \ref{LemmaSF2})}&{=}S(\lbrace j \rbrace, \lbrace j,l \rbrace)+S(\lbrace j \rbrace, T' \setminus \lbrace l \rbrace)+S(\lbrace l \rbrace, T' \setminus \lbrace j \rbrace)\\
\overset{Lem.~\ref{lemmacrucial} \ \ref{corcompl})}&{=}1-S(\lbrace l \rbrace, \lbrace j,l \rbrace)+S(\lbrace j \rbrace, T' \setminus \lbrace l \rbrace)+S(\lbrace l \rbrace, T' \setminus \lbrace j \rbrace)\\
\overset{Lem.~\ref{lemmacrucial} \ \ref{LemmaSF2})}&{=}1-S(\lbrace l \rbrace, T')+S(\lbrace j \rbrace, T' \setminus \lbrace l \rbrace)+2S(\lbrace l \rbrace, T' \setminus \lbrace j \rbrace)\\
&\equiv_2 1+S(\lbrace l \rbrace, T')+S(\lbrace j \rbrace, T' \setminus \lbrace l \rbrace),
\end{align*}
hence we must have
\[\sum_{j \in T'}(-1)^{S(\lbrace j \rbrace, T')}\frac{\gamma_i\gamma_j}{2\alpha}\eta_{T' \setminus  \lbrace j \rbrace}=0\]
for every $\emptyset \neq T' \subseteq Q$ with $|T'|$ even and every $i=0,1, \ldots, n$.
\end{proof}

\begin{remark}\label{rewritet0Q}
In the second case, that is when $\alpha \neq 0$, the elements $t_{0,Q}$ can be rewritten as
\[
\begin{split}
t_{0,Q}&=\sum_{\underset{|R| \equiv_2 |Q|}{R \subseteq Q}}(-1)^{S(Q \setminus R, Q)}\eta_{Q \setminus R} X_R+\\
&+\sum_{\underset{|R| \not\equiv_2 |Q|}{R \subseteq Q}}\sum_{i \in Q \setminus R}(-1)^{S(Q \setminus R, Q)+S(\lbrace i \rbrace, Q \setminus R)+1}\eta_{Q \setminus  (R \sqcup \lbrace i \rbrace)}\frac{\gamma_i}{2\alpha} GX_R.
\end{split}
\]
\end{remark}

\subsection{The elements \texorpdfstring{$t_{1, Q}$}{}}\label{sec:t1Q}

We describe the elements $t_{1,Q}$ for $|Q|\leq 1$.
Equality \eqref{B3cl} for $Q=\emptyset$ becomes 
\[
d_P(t_{k,\emptyset})=0 \quad \textrm{ for all } P\subseteq \lbrace 1, \ldots n \rbrace,  
\]
and in particular we have that $t_{1, \emptyset} \in \ker d_i$ for every $i=1, \ldots, n$. By Lemma~\ref{kerder} we have that $t_{1, \emptyset}=\mu_11+\mu_2 G$ for some $\mu_1, \mu_2 \in \Bbbk$. Since \eqref{B2cl} also holds, we immediately find $\mu_1=0$ and thus $t_{1, \emptyset}=\mu_2 G$.
It is straightforward to check that \eqref{B4cl} is satisfied, while \eqref{B5cl} gives
\[
t_{1,\emptyset}X_j =X_jt_{1,\emptyset}+2 t_{1,\lbrace i \rbrace} \quad \textrm{ for all } j=1, \ldots n,
\]
i.e.
\[t_{1,\lbrace j \rbrace}=\frac{\mu_2}{2}(GX_j -X_jG)=\frac{\mu_2}{2}(-\gamma_j+2GX_j) \quad \textrm{ for all } j=1, \ldots n.\]
Again, it is not hard to check that \eqref{B2cl}-\eqref{B4cl} are satisfied by such an element. Moreover \eqref{B5cl} with $k=1$, $Q=\lbrace j \rbrace$ and $i=j$ reads
\[t_{1,\lbrace j \rbrace}X_j =-X_jt_{1,\lbrace j \rbrace},\]
which is easily seen to hold.

Given $Q= \lbrace i_1<i_2<\ldots <i_{|Q|} \rbrace \subseteq  \lbrace 0, 1, \ldots, n \rbrace$ and a subset $R \subseteq Q$ with even cardinality, we can consider the set $\pi_2(R)$ of all perfect matchings for $R$. We write each perfect matching $\mathcal{P}$ for $R$ in his canonical form $\mathcal{P}=P_1|P_2|\ldots|P_m$, where $2m=|R|$, and we indicate with $\sgn\left(\mathcal{P}\right)$ the sign of $\mathcal{P}$ (see. Definition~\ref{sign}). For each block $P_k=\lbrace i_{k_1}<i_{k_2} \rbrace$ of $\mathcal{P}$ we define $\Lambda(\emptyset):=1$, $\Lambda(P_k):=\lambda_{i_{k_1}i_{k_2}}$ and for each $\mathcal{P}=P_1|P_2| \ldots |P_m$ we set $\Lambda(\mathcal{P}):=\prod_{k=1}^l \Lambda(P_k)$.

\begin{proposition}\label{propt1Q}
Let $X_0=G$, $\beta_0=\alpha$ and $\lambda_{0j}=\gamma_j$ for every $j=1, \ldots, n$. Let also $\overline{Q}:=Q \sqcup \lbrace 0 \rbrace$. For every $Q \subseteq \lbrace 1, \ldots, n \rbrace$ the element $t_{1,Q}$ satisfies \eqref{B5cl} if, and only if,
\begin{equation}\label{t1Q}
t_{1,Q}=\frac{\mu}{2^{\left\lfloor \frac{|\overline{Q}|}{2}\right\rfloor}}\sum_{\underset{|R| \equiv_2 |\overline{Q}|}{R \subseteq \overline{Q}}}\mathfrak{s}_Q(R)\sum_{\mathcal{P} \in \pi_2(\overline{Q}\setminus R)}\sgn(\mathcal{P})\Lambda(\mathcal{P})X_R,
\end{equation}
where $\mathfrak{s}_Q(R)=(-1)^{S(R,\overline{Q})+\frac{|\overline{Q}|-|R|}{2}}2^{\left\lfloor \frac{|R|}{2}\right\rfloor}$ and $\mu \in \Bbbk$.
\end{proposition}

\begin{proof}
We start by proving that if $t_{1,Q}$ satisfies \eqref{B5cl}, then it must be of the form \eqref{t1Q}. The proof is by induction on $|Q|$. For $Q=\emptyset$ we have
\begin{equation*}
t_{1,\emptyset}=\mu(-1)^{S(\overline{Q},\overline{Q})}\sgn(\emptyset)\Lambda(\emptyset)X_0=\mu G,
\end{equation*}
and for $Q=\lbrace i \rbrace$ we have
\begin{equation*}
t_{1,\lbrace i \rbrace}=\frac{\mu}{2}\left((-1)^{S(\emptyset,\overline{Q})+1}\sgn(\overline{Q})\Lambda(\overline{Q})+2(-1)^{S(\overline{Q},\overline{Q})}\sgn(\emptyset)\Lambda(\emptyset)X_0X_i\right)=\frac{\mu}{2}\left(-\gamma_i+2GX_i\right).
\end{equation*}
Both of these equalities hold true (see the beginning of this subsection).
Now suppose \eqref{t1Q} is verified for any $Q \subseteq \lbrace 1, \ldots, n \rbrace$ with cardinality $|Q|=m$ for a fixed $1 \leq m \leq n$. Consider $Q'=\lbrace i_1 < i_2 < \ldots <i_m < i_{m+1} \rbrace\subseteq \lbrace 1, \ldots, n \rbrace$ with cardinality $m+1$ and split it as $Q'=Q \sqcup \lbrace i_{m+1} \rbrace$, where $Q=\lbrace i_1 < i_2 < \ldots <i_m \rbrace$.
Then
\begin{align*}
t_{1,Q'}&=t_{1,Q \sqcup \lbrace i_{m+1} \rbrace}\\
\overset{\eqref{B5cl}}&{=}\frac{1}{2} \left(t_{1,Q}X_{i_{m+1}}+(-1)^{m+1}X_{i_{m+1}}t_{1,Q}\right)\\
\overset{Ind. \ hyp.}&{=} \frac{\mu}{2^{\left\lfloor \frac{|\overline{Q}|}{2}\right\rfloor+1}} \sum_{\underset{|R| \equiv_2 |\overline{Q}|}{R \subseteq \overline{Q}}}\mathfrak{s}_Q(R)\sum_{\mathcal{P} \in \pi_2(\overline{Q}\setminus R)}\sgn(\mathcal{P})\Lambda(\mathcal{P})\left(X_RX_{i_{m+1}}+(-1)^{|R|}X_{i_{m+1}}X_R\right)\\
\overset{\eqref{XRXj}}&{=} \frac{\mu}{2^{\left\lfloor \frac{|\overline{Q}|}{2}\right\rfloor+1}}\sum_{\underset{|R| \equiv_2 |\overline{Q}|}{R \subseteq \overline{Q}}}\mathfrak{s}_Q(R)\sum_{\mathcal{P} \in \pi_2(\overline{Q}\setminus R)}\sgn(\mathcal{P})\Lambda(\mathcal{P})2X_RX_{i_{m+1}}+\\
&+ \frac{\mu}{2^{\left\lfloor \frac{|\overline{Q}|}{2}\right\rfloor+1}}\sum_{\underset{|R| \equiv_2 |\overline{Q}|}{R \subseteq \overline{Q}}}\mathfrak{s}_Q(R)\sum_{\mathcal{P} \in \pi_2(\overline{Q}\setminus R)}\sgn(\mathcal{P})\Lambda(\mathcal{P})\sum_{i \in R} (-1)^{S(R \setminus \lbrace i \rbrace,R )+1}\lambda_{ii_{m+1}} X_{R \setminus \lbrace i \rbrace}\\
\overset{Lem.~\ref{lemmacrucial},~\ref{cor1})}&{=} \frac{\mu}{2^{\left\lfloor \frac{|\overline{Q'}|-1}{2}\right\rfloor}}\sum_{\underset{|R'| \equiv_2 |\overline{Q'}|}{i_{m+1} \in R' \subseteq \overline{Q'}}}(-1)^{S(R',\overline{Q'})+\frac{|\overline{Q'}|-|R'|}{2}}2^{\left\lfloor \frac{|R'|-1}{2}\right\rfloor}\sum_{\mathcal{P} \in \pi_2(\overline{Q'}\setminus R')}\sgn(\mathcal{P})\Lambda(\mathcal{P})X_{R'}+\\
&+ \frac{\mu}{2^{\left\lfloor \frac{|\overline{Q}|}{2}\right\rfloor+1}}\sum_{\underset{|R| \equiv_2 |\overline{Q}|}{R \subseteq \overline{Q}}}\mathfrak{s}_Q(R)\sum_{\mathcal{P} \in \pi_2(\overline{Q}\setminus R)}\sgn(\mathcal{P})\Lambda(\mathcal{P})\sum_{i \in R} (-1)^{S(R \setminus \lbrace i \rbrace,R )+1}\lambda_{ii_{m+1}} X_{R \setminus \lbrace i \rbrace}\\
&= \frac{\mu}{2^{\left\lfloor \frac{|\overline{Q'}|}{2}\right\rfloor}}\sum_{\underset{|R'| \equiv_2 |\overline{Q'}|}{i_{m+1} \in R' \subseteq \overline{Q'}}}\mathfrak{s}_{Q'}(R')\sum_{\mathcal{P} \in \pi_2(\overline{Q'}\setminus R')}\sgn(\mathcal{P})\Lambda(\mathcal{P})X_{R'}+\\
&+ \frac{\mu}{2^{\left\lfloor \frac{|\overline{Q'}|+1}{2}\right\rfloor}}\sum_{\underset{|R| \not\equiv_2 |\overline{Q'}|}{R \subseteq \overline{Q}}}\mathfrak{s}_Q(R)\sum_{\mathcal{P} \in \pi_2(\overline{Q}\setminus R)}\sgn(\mathcal{P})\Lambda(\mathcal{P})\sum_{i \in R} (-1)^{S(R \setminus \lbrace i \rbrace,R )+1}\lambda_{ii_{m+1}} X_{R \setminus \lbrace i \rbrace}\\
&= \frac{\mu}{2^{\left\lfloor \frac{|\overline{Q'}|}{2}\right\rfloor}}\sum_{\underset{|R'| \equiv_2 |\overline{Q'}|}{i_{m+1} \in R' \subseteq \overline{Q'}}}\mathfrak{s}_{Q'}(R')\sum_{\mathcal{P} \in \pi_2(\overline{Q'}\setminus R')}\sgn(\mathcal{P})\Lambda(\mathcal{P})X_{R'}+\frac{\mu}{2^{\left\lfloor \frac{|\overline{Q'}|+1}{2}\right\rfloor}} \cdot\\
&\cdot \sum_{\underset{|T| \equiv_2 |\overline{Q'}|}{i_{m+1} \notin T \subsetneq \overline{Q'}}}\sum_{i \in \overline{Q} \setminus T}\mathfrak{s}_{Q}(T \sqcup \lbrace i \rbrace)\sum_{\mathcal{P} \in \pi_2(U)}\sgn(\mathcal{P})\Lambda(\mathcal{P})(-1)^{S(T,T \sqcup \lbrace i \rbrace)+1}\lambda_{ii_{m+1}} X_{T},
\end{align*}
where $U=\overline{Q'}\setminus (T \sqcup \lbrace i, i_{m+1} \rbrace)=(\overline{Q'}\setminus T) \setminus \lbrace i, i_{m+1} \rbrace$.
Now we observe that
\begin{align*}
\mathfrak{s}_Q(T \sqcup \lbrace i \rbrace)(-1)^{S(T,T \sqcup \lbrace i \rbrace)+1}&=(-1)^{S(T \sqcup \lbrace i \rbrace,\overline{Q})+\frac{|\overline{Q}|-|T|-1}{2}}2^{\left\lfloor \frac{|T \sqcup \lbrace i \rbrace|}{2}\right\rfloor}(-1)^{S(T,T \sqcup \lbrace i \rbrace)+1}\\
\overset{Lem.~\ref{lemmacrucial}, \ \ref{lemmaFPF})}&{=}(-1)^{S(\lbrace i \rbrace,\overline{Q}\setminus T)+S(T,\overline{Q})+\frac{|\overline{Q}|-|T|-1}{2}+1}2^{\left\lfloor \frac{|T \sqcup \lbrace i \rbrace|}{2}\right\rfloor}\\
&=(-1)^{S(\lbrace i \rbrace,\overline{Q}\setminus T)}(-1)^{S(T,\overline{Q})+\frac{|\overline{Q}|+1-|T|}{2}}2^{\left\lfloor \frac{|T|+1}{2}\right\rfloor}\\
&=(-1)^{S(\lbrace i \rbrace,\overline{Q}\setminus T)}(-1)^{S(T,\overline{Q'})+\frac{|\overline{Q'}|-|T|}{2}}2^{\left\lfloor \frac{|T|+1}{2}\right\rfloor},
\end{align*}
and therefore
\begin{align*}
t_{1,Q'}&=\frac{\mu}{2^{\left\lfloor \frac{|\overline{Q'}|}{2}\right\rfloor}}\sum_{\underset{|R'| \equiv_2 |\overline{Q'}|}{i_{m+1} \in R' \subseteq \overline{Q'}}}\mathfrak{s}_{Q'}(R')\sum_{\mathcal{P} \in \pi_2(\overline{Q'}\setminus R')}\sgn(\mathcal{P})\Lambda(\mathcal{P})X_{R'}+\frac{\mu}{2^{\left\lfloor \frac{|\overline{Q'}|+1}{2}\right\rfloor}}\sum_{\underset{|T| \equiv_2 |\overline{Q'}|}{i_{m+1} \notin T \subsetneq \overline{Q'}}}\\
& \sum_{i \in \overline{Q} \setminus T}(-1)^{S(\lbrace i \rbrace,\overline{Q}\setminus T)}(-1)^{S(T,\overline{Q'})+\frac{|\overline{Q'}|-|T|}{2}}2^{\left\lfloor \frac{|T|+1}{2}\right\rfloor}\sum_{\mathcal{P} \in \pi_2(U)}\sgn(\mathcal{P})\Lambda(\mathcal{P})\lambda_{ii_{m+1}} X_{T}\\
&=\frac{\mu}{2^{\left\lfloor \frac{|\overline{Q'}|}{2}\right\rfloor}}\sum_{\underset{|R'| \equiv_2 |\overline{Q'}|}{i_{m+1} \in R' \subseteq \overline{Q'}}}\mathfrak{s}_{Q'}(R')\sum_{\mathcal{P} \in \pi_2(\overline{Q'}\setminus R')}\sgn(\mathcal{P})\Lambda(\mathcal{P})X_{R'}+\\
&+\frac{\mu}{2^{\left\lfloor \frac{|\overline{Q'}|}{2}\right\rfloor}}\sum_{\underset{|T| \equiv_2 |\overline{Q'}|}{i_{m+1} \notin T \subsetneq \overline{Q'}}}\mathfrak{s}_{Q'}(T) \sum_{i \in \overline{Q} \setminus T}\sum_{\mathcal{P} \in \pi_2(U)}(-1)^{S(\lbrace i \rbrace,\overline{Q}\setminus T)}\sgn(\mathcal{P})\Lambda(\mathcal{P})\lambda_{ii_{m+1}} X_{T}\\
\overset{Prop.~\ref{prop:pmrecur}}&{=}\frac{\mu}{2^{\left\lfloor \frac{|\overline{Q'}|}{2}\right\rfloor}}\sum_{\underset{|R'| \equiv_2 |\overline{Q'}|}{i_{m+1} \in R' \subseteq \overline{Q'}}}\mathfrak{s}_{Q'}(R')\sum_{\mathcal{P} \in \pi_2(\overline{Q'}\setminus R')}\sgn(\mathcal{P})\Lambda(\mathcal{P})X_{R'}+\\
&+\frac{\mu}{2^{\left\lfloor \frac{|\overline{Q'}|}{2}\right\rfloor}}\sum_{\underset{|T| \equiv_2 |\overline{Q'}|}{i_{m+1} \notin T \subsetneq \overline{Q'}}}\mathfrak{s}_{Q'}(T) \sum_{\mathcal{P'} \in \pi_2(\overline{Q'}\setminus T)}\sgn(\mathcal{P}')\Lambda(\mathcal{P}') X_{T}\\
&=\frac{\mu}{2^{\left\lfloor \frac{|\overline{Q'}|}{2}\right\rfloor}}\sum_{\underset{|R'| \equiv_2 |\overline{Q'}|}{R' \subseteq \overline{Q'}}}\mathfrak{s}_{Q'}(R')\sum_{\mathcal{P} \in \pi_2(\overline{Q'}\setminus R')}\sgn(\mathcal{P})\Lambda(\mathcal{P})X_{R'}.
\end{align*}
Thus the first implication is proved.

Now we can check that each $t_{1,Q}$ defined by \eqref{t1Q} satisfies \eqref{B5cl}, again by induction on $|Q|$. Since $t_{1,\emptyset}=\mu G$, it is easy to see that $t_{1,\emptyset}X_i-X_it_{1,\emptyset}=\mu G X_i-\mu X_i G=\mu(-\gamma_i+2GX_i)=2t_{1,\lbrace i \rbrace}$ for every $i=1, \ldots, n$. Similarly, if we consider $Q=\lbrace i \rbrace$, then
\[t_{1,\lbrace i \rbrace}X_i+X_it_{1,\lbrace i \rbrace}=\frac{\mu}{2}(-\gamma_iX_i+2\beta_i G-\gamma_iX_i+2X_i GX_i)=0\]
and
\begin{align*}
t_{1,\lbrace i \rbrace}X_j+X_jt_{1,\lbrace i \rbrace}&=\frac{\mu}{2}(-\gamma_iX_j+2GX_iX_j-\gamma_iX_j+2X_jGX_i)\\
&=\mu(-\gamma_iX_j+2GX_iX_j+\gamma_jX_i-\lambda_{ij}G)\\
&=2(-1)^{S(\lbrace i \rbrace, \lbrace i,j \rbrace)}t_{1,\lbrace i, j \rbrace}
\end{align*}
for every $j \neq i$.
Now suppose \eqref{B5cl} holds true for any $t_{1,Q}$ defined by \eqref{t1Q} for a $Q \subseteq \lbrace 1, \ldots, n \rbrace$ with cardinality $|Q|=m$ for a fixed $1 \leq m \leq n$. Consider $Q'=\lbrace i_1 < i_2 < \ldots <i_m < i_{m+1} \rbrace\subseteq \lbrace 1, \ldots, n \rbrace$ with cardinality $m+1$ and split it as $Q'=Q \sqcup \lbrace i_{m+1} \rbrace$, where $Q=\lbrace i_1 < i_2 < \ldots <i_m \rbrace$. Then
\begin{align*}
t_{1,Q'}X_{i_{m+1}}+(-1)^{|Q'|+1}X_{i_{m+1}}t_{1,Q'} \overset{Ind. \ hyp.}&{=}\frac{1}{2} \left(t_{1,Q}X_{i_{m+1}}X_{i_{m+1}}+(-1)^{|Q'|}X_{i_{m+1}}t_{1,Q}X_{i_{m+1}}\right)+\\
&+\frac{(-1)^{|Q'|+1}}{2} \left(X_{i_{m+1}}t_{1,Q}X_{i_{m+1}}+(-1)^{|Q'|}X_{i_{m+1}}X_{i_{m+1}}t_{1,Q}\right)\\
&=\frac{1}{2} \left(t_{1,Q}X_{i_{m+1}}X_{i_{m+1}}+(-1)^{|Q'|}X_{i_{m+1}}t_{1,Q}X_{i_{m+1}}\right)+\\
&+\frac{1}{2} \left(-(-1)^{|Q'|}X_{i_{m+1}}t_{1,Q}X_{i_{m+1}}-X_{i_{m+1}}X_{i_{m+1}}t_{1,Q}\right)\\
&=0.
\end{align*}
Similarly
\begin{align*}
&t_{1,Q'}X_i+(-1)^{|Q'|+1}X_it_{1,Q'}\\
\overset{Ind. \ hyp.}&{=}\frac{1}{2} \left(t_{1,Q}X_{i_{m+1}}X_i+(-1)^{|Q'|}X_{i_{m+1}}t_{1,Q}X_i\right)\\
&+\frac{(-1)^{|Q'|+1}}{2} \left(X_it_{1,Q}X_{i_{m+1}}+(-1)^{|Q'|}X_iX_{i_{m+1}}t_{1,Q}\right)\\
&=\frac{1}{2} \left(\lambda_{ii_{m+1}}t_{1,Q}-t_{1,Q}X_iX_{i_{m+1}}+(-1)^{|Q'|}X_{i_{m+1}}t_{1,Q}X_i\right)+\\
&-\frac{1}{2}\left((-1)^{|Q'|}X_it_{1,Q}X_{i_{m+1}}+\lambda_{ii_{m+1}}t_{1,Q}-X_{i_{m+1}}X_it_{1,Q}\right)\\
&=\frac{1}{2} \left[-\left(t_{1,Q}X_i+(-1)^{|Q'|}X_it_{1,Q}\right)X_{i_{m+1}}+(-1)^{|Q'|}X_{i_{m+1}}\left(t_{1,Q}X_i+(-1)^{|Q'|}X_it_{1,Q}\right)\right]\\
\end{align*}
for every $i \neq i_{m+1}$. If $i \in Q$ then this whole sum vanishes thanks to the induction hypothesis. On the other hand, if $i \notin Q'$ we either have $i>i_{m+1}$ or $i_m<i<i_{m+1}$.
In the first case one explicitly writes $t_{1,Q \sqcup \lbrace i \rbrace}$ using \eqref{t1Q} and then proves that
\[t_{1,Q'}X_i+(-1)^{|Q'|+1}X_it_{1,Q'}=2t_{1,Q' \sqcup \lbrace i \rbrace}=2(-1)^{S(Q',Q'\sqcup \lbrace i \rbrace)}t_{1,Q' \sqcup \lbrace i \rbrace}\]
with the same calculations used in the first part to show that \eqref{B5cl} implies \eqref{t1Q}. If instead $i_m<i<i_{m+1}$, we can first observe that
\begin{align*}
&t_{1,Q'}X_i+(-1)^{|Q'|+1}X_it_{1,Q'}=\\
&=\frac{1}{2} \left[-\left(t_{1,Q}X_i+(-1)^{|Q'|}X_it_{1,Q}\right)X_{i_{m+1}}+(-1)^{|Q'|}X_{i_{m+1}}\left(t_{1,Q}X_i+(-1)^{|Q'|}X_it_{1,Q}\right)\right]\\
&=(-1)^{S(Q,Q\sqcup \lbrace i \rbrace)+1}\left[t_{1,Q\sqcup \lbrace i \rbrace}X_{i_{m+1}}+(-1)^{|Q'|+1}X_{i_{m+1}}t_{1,Q\sqcup \lbrace i \rbrace}\right]\\
\end{align*}
and then prove once again directly that this quantity is equal to $2(-1)^{S(Q,Q\sqcup \lbrace i \rbrace)+1}t_{1,Q \sqcup \lbrace i, i_{m+1}\rbrace}$. Finally observe that for $i_m<i<i_{m+1}$ we get 
\[2(-1)^{S(Q,Q\sqcup \lbrace i \rbrace)+1}t_{1,Q \sqcup \lbrace i, i_{m+1}\rbrace}=-2t_{1,Q \sqcup \lbrace i, i_{m+1}\rbrace}=2(-1)^{S(Q',Q'\sqcup \lbrace i \rbrace)}t_{1,Q' \sqcup \lbrace i \rbrace}.\]
\end{proof}

\begin{proposition}\label{propt1Q2}
For every $Q \subseteq \lbrace 1, \ldots, n \rbrace$ the elements $t_{1,Q}$ defined in \eqref{t1Q} satisfy \eqref{B2cl}, \eqref{B3cl} and \eqref{B4cl}.
\end{proposition}

\begin{proof}
We have
\begin{align*}
\sigma(t_{1,Q})&=\frac{\mu}{2^{\left\lfloor \frac{|\overline{Q}|}{2}\right\rfloor}}\sum_{\underset{|R| \equiv_2 |\overline{Q}|}{R \subseteq \overline{Q}}}\mathfrak{s}_Q(R)\sum_{\mathcal{P} \in \pi_2(\overline{Q}\setminus R)}\sgn(\mathcal{P})\Lambda(\mathcal{P})\sigma(X_R)\\
&=\frac{\mu}{2^{\left\lfloor \frac{|\overline{Q}|}{2}\right\rfloor}}\sum_{\underset{|R| \equiv_2 |\overline{Q}|}{R \subseteq \overline{Q}}}\mathfrak{s}_Q(R)\sum_{\mathcal{P} \in \pi_2(\overline{Q}\setminus R)}\sgn(\mathcal{P})\Lambda(\mathcal{P})(-1)^{|R|}X_R\\
&=\frac{\mu}{2^{\left\lfloor \frac{|\overline{Q}|}{2}\right\rfloor}}\sum_{\underset{|R| \equiv_2 |\overline{Q}|}{R \subseteq \overline{Q}}}\mathfrak{s}_Q(R)\sum_{\mathcal{P} \in \pi_2(\overline{Q}\setminus R)}\sgn(\mathcal{P})\Lambda(\mathcal{P})(-1)^{|Q|+1}X_R\\
&=(-1)^{|Q|+1}\frac{\mu}{2^{\left\lfloor \frac{|\overline{Q}|}{2}\right\rfloor}}\sum_{\underset{|R| \equiv_2 |\overline{Q}|}{R \subseteq \overline{Q}}}\mathfrak{s}_Q(R)\sum_{\mathcal{P} \in \pi_2(\overline{Q}\setminus R)}\sgn(\mathcal{P})\Lambda(\mathcal{P})X_R\\
&=(-1)^{|Q|+1}t_{1,Q},
\end{align*}
so \eqref{B2cl} is satisfied.
Next 
\[d_P(t_{1,Q})=\frac{\mu}{2^{\left\lfloor \frac{|\overline{Q}|}{2}\right\rfloor}}\sum_{\underset{|R| \equiv_2 |\overline{Q}|}{R \subseteq \overline{Q}}}2^{\left\lfloor \frac{|R|}{2}\right\rfloor}\sum_{\mathcal{P} \in \pi_2(\overline{Q}\setminus R)}(-1)^{\alpha(\mathcal{P},\overline{Q})}\Lambda(\mathcal{P})d_P(X_R).\]
If $P \not\subseteq Q$, then $P \not\subseteq R$ for every $R$ appearing in the above summation and thus $d_P(t_{1,Q})=0$, by Lemma~\ref{lemmader2}. On the other hand, if $P \subseteq Q$, then $P$ is contained in some $R \subseteq Q$ and we can write
\begin{align*}
&d_P(t_{1,Q})=\frac{\mu}{2^{\left\lfloor \frac{|\overline{Q}|}{2}\right\rfloor}}\sum_{\underset{|R| \equiv_2 |\overline{Q}|}{R \subseteq \overline{Q}}}\mathfrak{s}_Q(R)\sum_{\mathcal{P} \in \pi_2(\overline{Q}\setminus R)}\sgn(\mathcal{P})\Lambda(\mathcal{P})d_P(X_R)\\
\overset{Lem.~\ref{lemmader2}}&{=}\frac{\mu}{2^{\left\lfloor \frac{|\overline{Q}|}{2}\right\rfloor}}\sum_{\underset{|R| \equiv_2 |\overline{Q}|}{P \subseteq R \subseteq \overline{Q}}}\mathfrak{s}_Q(R)\sum_{\mathcal{P} \in \pi_2(\overline{Q}\setminus R)}\sgn(\mathcal{P})\Lambda(\mathcal{P})(-1)^{S(P, R)+\frac{|P|(|P|+1)}{2}}X_{R \setminus P}\\
&=\frac{\mu}{2^{\left\lfloor \frac{|\overline{Q}|}{2}\right\rfloor}}\sum_{\underset{|R'| \equiv_2 |\overline{Q}\setminus P|}{R' \subseteq \overline{Q}\setminus P}}\mathfrak{s}_Q(R' \sqcup P)\sum_{\mathcal{P} \in \pi_2(\overline{Q}\setminus (R' \sqcup P))}\sgn(\mathcal{P})\Lambda(\mathcal{P})(-1)^{S(P, R' \sqcup P)+\frac{|P|(|P|+1)}{2}}X_{R'}\\
\overset{Lem.~\ref{lemmacrucial},~\ref{lemmaFPF})}&{=}\frac{\mu}{2^{\left\lfloor \frac{|\overline{Q}\setminus P|}{2}\right\rfloor}}\sum_{\underset{|R'| \equiv_2 |\overline{Q}\setminus P|}{R' \subseteq \overline{Q}\setminus P}}(-1)^{S(P, \overline{Q})}\mathfrak{s}_{Q\setminus P}(R')\sum_{\mathcal{P} \in \pi_2(\overline{Q}\setminus (R' \sqcup P))}\sgn(\mathcal{P})\Lambda(\mathcal{P})(-1)^{\frac{|P|(|P|+1)}{2}}X_{R'}\\
&=(-1)^{S(P, \overline{Q})+\frac{|P|(|P|+1)}{2}}\frac{\mu}{2^{\left\lfloor \frac{|\overline{Q}\setminus P|}{2}\right\rfloor}}\sum_{\underset{|R'| \equiv_2 |\overline{Q}\setminus P|}{R' \subseteq \overline{Q}\setminus P}}\mathfrak{s}_{Q\setminus P}(R')\sum_{\mathcal{P} \in \pi_2(\overline{Q}\setminus (R' \sqcup P))}\sgn(\mathcal{P})\Lambda(\mathcal{P})X_{R'}\\
&=(-1)^{S(P, \overline{Q})+\frac{|P|(|P|+1)}{2}}t_{1,Q\setminus P}\\
&=(-1)^{S(P, Q)+|P|+\frac{|P|(|P|+1)}{2}}t_{1,Q\setminus P}.
\end{align*}
This proves that \eqref{B3cl} holds. To verify that $t_{1,Q}$ satisfy \eqref{B4cl} we can proceed by induction on $|Q|$. If $Q=\emptyset$, then we have
\[t_{1,\emptyset}G=\mu GG=\mu\alpha=G\mu G=(-1)^{0}Gt_{1,\emptyset},\]
while if $Q=\lbrace i \rbrace$ we see that
\[t_{1,\lbrace i \rbrace}G=\frac{\mu}{2}(-\gamma_i+2GX_i)G=\frac{\mu}{2}(\gamma_iG-2\alpha X_i)=-\frac{\mu}{2}G(-\gamma_i+2GX_i)=(-1)^{|\lbrace i \rbrace|}Gt_{1,\lbrace i \rbrace}.\]
Now suppose \eqref{B4cl} holds true for any $t_{1,Q}$ defined in \eqref{t1Q} for a $Q \subseteq \lbrace 1, \ldots, n \rbrace$ with cardinality $|Q|=m$ for a fixed $1 \leq m \leq n$. Consider $Q'=\lbrace i_1 < i_2 < \ldots <i_m < i_{m+1} \rbrace\subseteq \lbrace 1, \ldots, n \rbrace$ with cardinality $m+1$ and split it as $Q'=Q \sqcup \lbrace i_{m+1} \rbrace$, where $Q=\lbrace i_1 < i_2 < \ldots <i_m \rbrace$. Remember that $t_{1,Q'}$ satisfies \eqref{B5cl} (see Prop.~\ref{propt1Q}), therefore
\[t_{1,Q'}=\frac{1}{2} \left(t_{1,Q}X_{i_{m+1}}+(-1)^{|Q'|}X_{i_{m+1}}t_{1,Q}\right)\]
Then
\begin{align*}
&t_{1,Q'}G +(-1)^{|Q'|+1}Gt_{1,Q'}\\
&=\frac{1}{2} \left[t_{1,Q}X_{i_{m+1}}G+(-1)^{|Q'|}X_{i_{m+1}}t_{1,Q}G+(-1)^{|Q'|+1}G t_{1,Q}X_{i_{m+1}}-GX_{i_{m+1}}t_{1,Q}\right]\\
&=\frac{1}{2} \left[-t_{1,Q}GX_{i_{m+1}}+(-1)^{|Q'|}X_{i_{m+1}}t_{1,Q}G+(-1)^{|Q'|+1}G t_{1,Q}X_{i_{m+1}}+X_{i_{m+1}}Gt_{1,Q}\right]\\
&=\frac{1}{2} \left[X_{i_{m+1}}\left(Gt_{1,Q}+(-1)^{|Q'|}t_{1,Q}G\right)-\left( t_{1,Q}G+(-1)^{|Q'|}G t_{1,Q} \right)X_{i_{m+1}}\right].
\end{align*}
The induction hypothesis reads $t_{1,Q}G+(-1)^{|Q|+1}G t_{1,Q}=0$, so we can conclude that also $t_{1,Q'}G +(-1)^{|Q'|+1}Gt_{1,Q'}=0$ and we are done.
\end{proof}

\subsection{The main results}

\begin{theorem}\label{theo3}
Let $H=E(n)$ and let $A=Cl(\alpha,\beta_i, \gamma_i, \lambda_{ij})$ be a Clifford algebra endowed with the canonical $H$-coaction defined by \eqref{canonical1}-\eqref{canonicalGXi}. The cowreath $(A\otimes H^{op},H,\psi )$ is rt-separable if, and only if,
\begin{enumerate}
\item $A=Cl(0,0,0,0)$ \\
or
\item $A=Cl\left(\alpha,\frac{\gamma_i^2}{4\alpha}, \gamma_i, \frac{\gamma_i\gamma_j}{2 \alpha}\right)$ with $\alpha \neq 0$.  
\end{enumerate}
Let $\overline{Q}=Q \sqcup \lbrace 0 \rbrace$. The Casimir element is given by
\begin{equation}
B^A(g^jx_P \otimes g^kx_Q)=B^A(1 \otimes S(g^jx_P)g^kx_Q) \quad \text{and} \quad B^{A}\left(
1_{H}\otimes 1_{H}\right) =1_{A},  \label{eq:1}
\end{equation}%
\begin{equation}\label{eq:2}
B^A(1 \otimes x_Q)=\sum_{\underset{|R| \equiv_2 |Q|}{R \subseteq \overline{Q}}}(-1)^{S(Q \setminus R, Q)}\eta_{Q \setminus R} X_R,
\end{equation}
\begin{equation}\label{eq:3}
B^A(1 \otimes gx_Q)=\frac{\mu}{2^{\left\lfloor \frac{|\overline{Q}|}{2}\right\rfloor}}\sum_{\underset{|R| \equiv_2 |\overline{Q}|}{R \subseteq \overline{Q}}}(-1)^{S(R,\overline{Q})+\frac{|\overline{Q}|-|R|}{2}}2^{\left\lfloor \frac{|R|}{2}\right\rfloor}\sum_{\mathcal{P} \in \pi_2(\overline{Q}\setminus R)}\sgn(\mathcal{P})\Lambda(\mathcal{P})X_R,  
\end{equation}
for $j,k=0,1$ and every $P,Q \subseteq \lbrace 1, \ldots, n \rbrace$, where $\mu \in \Bbbk$ and $\eta_S \in \Bbbk$ for every $S \subseteq \lbrace 1, \ldots, n \rbrace$. If $\alpha \neq 0$, then $\eta_T =\sum_{i \in T}(-1)^{S(\lbrace i \rbrace, T)+1}\eta_{T \setminus  \lbrace i \rbrace}\frac{\gamma_i}{2\alpha}$ for every $T \subseteq Q$, with $|T|$ odd and every $Q \subseteq \lbrace 1, \ldots, n \rbrace$. 
\end{theorem}

\begin{proof}
By Theorem~\ref{theo2} we have that the cowreath $(A\otimes E(n)^{op},E(n),\psi )$ is separable via a
Casimir element%
\begin{equation*}
B:E(n) \otimes E(n) \rightarrow A\otimes E(n)^{op}
\end{equation*}%
of the form
\begin{equation*}
h\otimes h^{\prime }\rightarrow B^{A}\left( h\otimes h^{\prime }\right)
\otimes 1_{H}
\end{equation*}%
if, and only if, the equalities \eqref{B0cl}-\eqref{B5cl} hold for $j,k=0,1$ and for every $P,Q, \lbrace i \rbrace \subseteq \lbrace 1, \ldots, n\rbrace$. If we suppose they do, then \eqref{B0cl} and \eqref{B1cl} coincide with \eqref{eq:1}, while \eqref{B2cl} and \eqref{B3cl} imply \eqref{eq:2}, by Proposition~\ref{propt0Q}. Equality \eqref{eq:3} is implied by \eqref{B5cl}, thanks to Proposition~\ref{propt1Q}, and the conditions on the defining scalars and the $\eta_T$'s follow from \eqref{B4cl} and \eqref{B5cl} as proved in Proposition~\ref{propt0Q2}.

Conversely suppose that \eqref{eq:1}-\eqref{eq:3} hold true and that we are either in case \emph{(1)} or \emph{(2)}. Then \eqref{eq:1} is the same as \eqref{B0cl} and \eqref{B1cl}, \eqref{eq:3} implies \eqref{B2cl}-\eqref{B5cl} for $k=1$, by Proposition~\ref{propt1Q} and Proposition~\ref{propt1Q2}. \eqref{eq:2} implies \eqref{B2cl} and \eqref{B3cl} for $k=0$, thanks to Proposition~\ref{propt0Q}, and since either \emph{(1)} or \emph{(2)} is verified we also have that \eqref{B4cl} and \eqref{B5cl} hold for $k=0$, by Proposition~\ref{propt0Q2}. We can conclude, by Theorem~\ref{theo2}, that the cowreath $(A\otimes H^{op},H,\psi )$ is rt-separable.
\end{proof}

\begin{theorem}\label{theo4}
Let $H=E(n)$ and let $A=Cl(\alpha,\beta_i, \gamma_i, \lambda_{ij})$ be a Clifford algebra endowed with the canonical $H$-coaction defined by \eqref{canonical1}-\eqref{canonicalGXi}. The cowreath $(A\otimes H^{op},H,\psi )$ is rth-separable if, and only if, $A=Cl\left(\alpha,\frac{\gamma_i^2}{4\alpha}, \gamma_i, \frac{\gamma_i\gamma_j}{2 \alpha}\right)$ with $\alpha \neq 0$. The Casimir element is given by 
\begin{equation}
B^A(g^jx_P \otimes g^kx_Q)=B^A(1 \otimes S(g^jx_P)g^kx_Q),  \label{eq:4}
\end{equation}%
\begin{equation}\label{eq:5}
\begin{split}
B^A(1 \otimes x_Q)&=\sum_{\underset{|R| \equiv_2 |Q|}{R \subseteq Q}}(-1)^{S(Q \setminus R, Q)+\frac{|Q|-|R|}{2}}\frac{\prod_{j \in Q\setminus R}\gamma_j}{2^{|Q|-|R|}\alpha^{\frac{|Q|-|R|}{2}}} X_R+\\
&+\sum_{\underset{|R| \not\equiv_2 |Q|}{R \subseteq Q}}(-1)^{S(Q \setminus R, Q)+\frac{|Q|-|R|+1}{2}}\frac{\prod_{j \in Q\setminus R}\gamma_j}{2^{|Q|-|R|} \alpha^{\frac{|Q|-|R|+1}{2}}} GX_R,
\end{split}
\end{equation}
\begin{equation}\label{eq:6}
\begin{split}
B^A(1 \otimes gx_Q)&=(\mu G) B^A(1 \otimes x_Q)
\end{split}
\end{equation}
for $j,k=0,1$ and every $P,Q \subseteq \lbrace 1, \ldots, n \rbrace$,
with $\mu \in \Bbbk$ and $\mu^2\alpha=1$.
\end{theorem}

\begin{proof}
If the cowreath is rth-separable then it is rt-separable and \eqref{B6cl} holds. The first equality of  \eqref{B6cl} reads $\mu^2 \alpha=(\mu G)^2=1$ which implies that $\mu, \alpha \neq 0$. Then, by \autoref{theo3}, we get that $A=Cl\left(\alpha,\frac{\gamma_i^2}{4\alpha}, \gamma_i, \frac{\gamma_i\gamma_j}{2 \alpha}\right)$ and that \eqref{eq:4} holds, together with
\begin{align*}
t_{0,Q} \overset{\eqref{B6cl}}&{=}(-1)^{|Q|}t_{1,Q}t_{1,\emptyset}=(-1)^{|Q|}\mu t_{1,Q}G\overset{\eqref{B4cl}}{=}\mu Gt_{1,Q}\\
\overset{\eqref{eq:3}}&{=}\frac{\mu^2}{2^{\left\lfloor \frac{|\overline{Q}|}{2}\right\rfloor}}\sum_{\underset{|R| \equiv_2 |\overline{Q}|}{R \subseteq \overline{Q}}}(-1)^{S(R,\overline{Q})+\frac{|\overline{Q}|-|R|}{2}}2^{\left\lfloor \frac{|R|}{2}\right\rfloor}\sum_{\mathcal{P} \in \pi_2(\overline{Q}\setminus R)}\sgn(\mathcal{P})\Lambda(\mathcal{P})GX_R.
\end{align*}
Notice that since $\lambda_{ij}=\frac{\gamma_i\gamma_j}{2 \alpha}$ for every $i, j \in \lbrace 1, \ldots, n \rbrace$ (with the convention that $\lambda_{ii}=2\beta_i$), then $\Lambda(\mathcal{P})= \frac{1}{(2\alpha)^{\frac{|U|}{2}}}\prod_{j \in U} \gamma_j$ for every $\mathcal{P} \in \pi_2(U)$ and every $U$ with even cardinality. Thanks to Proposition~\ref{prop:matching}, $\ref{prop:sgn})$, we can conclude that
\begin{align*}
t_{0,Q} &=\frac{\mu^2}{2^{\left\lfloor \frac{|\overline{Q}|}{2}\right\rfloor}}\sum_{\underset{|R| \equiv_2 |\overline{Q}|}{R \subseteq \overline{Q}}}(-1)^{S(R,\overline{Q})+\frac{|\overline{Q}|-|R|}{2}}2^{\left\lfloor \frac{|R|}{2}\right\rfloor}\sum_{\mathcal{P} \in \pi_2(\overline{Q}\setminus R)}\sgn(\mathcal{P})\Lambda(\mathcal{P})GX_R\\
\overset{Prop.~\ref{prop:matching}, \ \ref{prop:sgn})}&{=}\frac{\mu^2}{2^{\left\lfloor \frac{|\overline{Q}|}{2}\right\rfloor}}\sum_{\underset{|R| \equiv_2 |\overline{Q}|}{R \subseteq \overline{Q}}}(-1)^{S(R,\overline{Q})+\frac{|\overline{Q}|-|R|}{2}}2^{\left\lfloor \frac{|R|}{2}\right\rfloor}\frac{\prod_{j \in \overline{Q}\setminus R} \gamma_j}{(2\alpha)^{\frac{|\overline{Q}|-|R|}{2}}}GX_R\\
&=\frac{1}{2^{\left\lfloor \frac{|Q|+1}{2}\right\rfloor}}\sum_{\underset{|R| \equiv_2 |Q|}{R \subseteq Q}}(-1)^{S(R,Q)+\frac{|Q|-|R|}{2}}2^{\left\lfloor \frac{|R|+1}{2}\right\rfloor}\frac{\prod_{j \in Q \setminus R} \gamma_j}{(2\alpha)^{\frac{|Q|-|R|}{2}}}X_R\\
&+\frac{\mu^2}{2^{\left\lfloor \frac{|Q|+1}{2}\right\rfloor}}\sum_{\underset{|R| \not\equiv_2 |Q|}{R \subseteq Q}}(-1)^{S(R,Q)+|R|+\frac{|Q|-|R|+1}{2}}2^{\left\lfloor \frac{|R|}{2}\right\rfloor}\frac{\prod_{j \in Q \setminus R} \gamma_j}{(2\alpha)^{\frac{|Q|-|R|-1}{2}}}GX_R\\
\overset{Lem.~\ref{lemmacrucial}, \ \ref{corcompl})}&{=}\sum_{\underset{|R| \equiv_2 |Q|}{R \subseteq Q}}(-1)^{S(Q \setminus R,Q)+\frac{|Q|-|R|}{2}}\frac{\prod_{j \in Q \setminus R} \gamma_j}{2^{|Q|-|R|}\alpha^{\frac{|Q|-|R|}{2}}}X_R+\\
&+\sum_{\underset{|R| \not\equiv_2 |Q|}{R \subseteq Q}}(-1)^{S(Q \setminus R,Q)+\frac{|Q|-|R|+1}{2}}\frac{\prod_{j \in Q \setminus R} \gamma_j}{2^{|Q|-|R|-1}\alpha^{\frac{|Q|-|R|+1}{2}}}GX_R
\end{align*}
and thus we find \eqref{eq:5}. Equality \eqref{eq:6} follows from \eqref{B6cl} and \eqref{B4cl}. Conversely, one can check that if equalities \eqref{eq:4}-\eqref{eq:6} are satisfied then \eqref{B0cl}-\eqref{B6cl} hold true.
\end{proof}

\end{document}